\documentclass[a4paper,11pt,DIV=12,abstracton]{scrartcl}
\usepackage[utf8]{inputenc}
\usepackage[T1]{fontenc}
\usepackage{amsmath, amssymb, amsfonts, amsbsy, latexsym,amsthm}
\usepackage[english]{babel}
\usepackage{color,graphicx}

\setkomafont{disposition}{\normalfont\bfseries}

\usepackage{enumerate}

\usepackage{tikz}
\usetikzlibrary{arrows}
\usetikzlibrary{positioning}

\usepackage{booktabs}
\usepackage{multirow}

\DeclareMathOperator{\Sp}{Sp} 

\DeclareMathOperator{\Sym}{Sym}
\DeclareMathOperator{\GL}{GL}
\DeclareMathOperator{\TDS}{TDS}

\DeclareMathOperator{\sgn}{sgn}
\DeclareMathOperator{\supp}{supp}

\DeclareMathOperator{\Co}{Co}

\DeclareMathOperator{\ad}{Ad}

\newcommand{\msharp}{\mathbin{\sharp}}

\newtheorem{theorem}{Theorem}[section]
\newtheorem{lemma}[theorem]{Lemma}

\newtheorem{corollary}[theorem]{Corollary}

\newtheorem{remark}[theorem]{Remark}

\newcommand{\cA}{\mathcal{A}}
\newcommand{\cB}{\mathcal{B}}
\newcommand{\cF}{\mathcal{F}}

\newcommand{\cH}{\mathcal{H}}
\newcommand{\cI}{\mathcal{I}}
\newcommand{\cJ}{\mathcal{J}}

\newcommand{\cR}{\mathcal{R}}

\newcommand{\cU}{\mathcal{U}}
\newcommand{\cW}{\mathcal{W}}
\newcommand{\cSH}{\mathcal{SH}}
\newcommand{\cSC}{\mathcal{SC}}

\newcommand{\ZZ}{\mathbb{Z}}

\newcommand{\RR}{\mathbb{R}}

\let\SS\undefined

\newcommand{\HH}{\mathbb{H}}
\newcommand{\SS}{\mathbb{S}}

\newcommand{\ga}{\mathfrak{a}}
\renewcommand{\gg}{\mathfrak{g}}
\newcommand{\gh}{\mathfrak{h}}
\newcommand{\sptwor}{\mathfrak{sp}(2,\RR)}

\newcommand{\eps}{\varepsilon}

\def\tT{{\mbox{\tiny{T}}}}

\begin{document}
\title{Different faces of the shearlet group}
\author{
  Stephan Dahlke
  \thanks{
    FB12 Mathematik und Informatik, 
    Philipps-Universität Marburg, 
    Hans-Meerwein Straße,
    Lahnberge, 
    35032 Marburg, 
    Germany, 
    dahlke@mathematik.uni-kl.de
  },
\and
  Filippo De Mari\thanks{
    Dipartimento di Matematica,
    Universit\`a degli Studi di Genova,
    Via Dodecaneso 35,
    16146 Genova, 
    Italy,
    \{demari,devito\}@dima.unige.it
  }, 
\and
  Ernesto De Vito\footnotemark[2],
\and
  Sören Häuser
  \thanks{
    Fachbereich für Mathematik, 
    Technische Universität Kaiserslautern, 
    Paul-Ehrlich-Str. 31,
    67663 Kaiserslautern, 
    Germany,
    \{haeuser,steidl\}@mathematik.uni-kl.de  
  },
\and
  Gabriele Steidl\footnotemark[3],
\and
  Gerd Teschke
  \thanks{
    Institute for Computational Mathematics in Science and Technology, 
    Hochschule Neubrandenburg,
    University of Applied Sciences, 
    Brodaer Str. 2, 
    17033 Neubrandenburg, 
    Germany,
    teschke@hs-nb.de
  }
}
\date{\today}

\maketitle

\begin{abstract}
Recently, shearlet groups have received much attention in connection with shearlet transforms
applied for orientation sensitive image analysis and restoration. 
The square integrable representations of the shearlet groups provide not only the basis for the shearlet transforms 
but also for a very natural definition of scales of smoothness spaces, called shearlet coorbit spaces.
The aim of this paper is twofold: first we discover isomorphisms between shearlet groups
and other well-known groups, namely extended Heisenberg groups and subgroups of the symplectic group.
Interestingly, the connected shearlet group with positive dilations has an isomorphic copy in the symplectic group,
while this is not true for the full shearlet group with all nonzero dilations. 
Indeed we prove the general result that there exist, up to adjoint action of the symplectic group,
only one embedding of the extended Heisenberg algebra into the Lie algebra of the symplectic group. 

Having understood the various group isomorphisms 
it is natural to ask for the relations between coorbit spaces of isomorphic groups 
with equivalent representations. These connections are examined in the second part
of the paper. We describe how isomorphic groups 
with equivalent representations lead to isomorphic coorbit spaces.
In particular we apply this result to square integrable representations of the connected shearlet groups and metaplectic representations of
subgroups of the symplectic group.
This implies the definition of metaplectic coorbit spaces.

Besides the usual full and connected shearlet groups we also deal with Toeplitz shearlet groups.
\end{abstract}


\textbf{Mathematics Subject Classifiation}
22D10, 
22E30, 
22E60, 
42B35,
42C15

%
\section{Introduction}\label{sec1}
%
The shearlet transform was originally developed in the inaugural paper \cite{LWWW02}.
Among other transforms applied in the analysis of directional information,
the continuous shearlet transform stands out because it is related to group theory, i.e., 
it can be derived from square integrable representations of the so-called shearlet group \cite{DKMSST08}.
Recently, the shearlet transform and its modifications have found wide applications in the analysis and restoration of images, see, e.g., \cite{Gro11a,GL09,GLL09,HS13,KKL13}.
Further, the shearlet group and its square integrable representations give rise to a natural scale of smoothness spaces, 
so-called shearlet coorbit spaces \cite{DHST13,DKST09,DST10} which can be considered as a special example of general coorbit spaces
introduced by Feichtinger and Gr\"ochenig \cite{FG88,FG89a,FG89b}.
An overview of recent developments on shearlets can be found in the book \cite{KL12a}.
On the other hand, other groups such as the Heisenberg group or the symplectic group have been playing  an important role in harmonic analysis, linear algebra and signal and image processing
for a long time. 
In particular, the Heisenberg group is one of the basic tools  for the
mathematical foundation of the short-time Fourier transform, see \cite{FG89a,Gro01} for details. 
For the symplectic group we refer to \cite{CDNT06a,CDNT06b,CDNT10} and references therein.

In this paper we ask for relations between the different groups and the corresponding coorbit spaces.
In particular, we discover isometries of the  connected shearlet group and its relatives 
to extended Heisenberg groups and subgroups of the symplectic group. Interestingly such results
do not hold true for the full shearlet group.
We show that isomorphic groups with equivalent representations give rise to equivalent coorbit spaces.

The first observation concerning the relationship between the extended Heisenberg group and the shearlet group is found in \cite{ST99}, 
where square integrability is shown, and it is recalled in the book \cite{KT13}. 
The first embedding of the two-dimensional continuous shearlet group into $\Sp(2,\RR)$ is found in \cite{Kin09}, 
and, with $\gamma=1$, implicitly appears in \cite{CDNT06a}. 
Similarly, the first embedding of shearlet-like groups with isotropic dilations in arbitrary dimensions (albeit with non-Toeplitz shearing matrices) 
into $\Sp(2,\RR)$ with matrices of the form $\Sigma\rtimes H$ is found in \cite{Kin09}. 
The isotropic shearlet-like groups were further explored in the paper \cite{CK12}, 
while the anisotropic case (including the embedding into $\Sp(2,\RR)$) was generalized to higher dimensions in \cite{CK14}.

\emph{Organization of the paper}:
In Section~\ref{sec2} we discover isomorphisms between shearlet and Toeplitz shearlet groups and other well-known groups. 
We show that the full and connected shearlet groups are isomorphic to full and connected extended Heisenberg groups, respectively.
Further, we prove that the connected shearlet group is isomorphic to a subgroup of the symplectic group, which 
holds also true for the connected Toeplitz shearlet group.  
Section~\ref{sec3} deals with the full shearlet group. We show that it is not possible to embed this group into the symplectic group or in any of its coverings. 
The proof is based on the general result that the Lie algebra of the extended Heisenberg group can only be embedded in one way into the Lie algebra of the symplectic group. 
This result stands also for its own and is presented in Section~\ref{sec4}. 
Finally, in Section~\ref{sec5} the very natural relations between coorbit spaces of isomorphic groups with equivalent representations are presented. 
Naturally, these coorbit spaces are also isomorphic. 
At the end of Section~\ref{sec5} we use our findings to introduce metaplectic coorbit spaces.
%
\section{Shearlet groups and their isomorphic relatives} \label{sec2}
%
In this section we show that the shearlet groups are isomorphic to extended Heisenberg groups and
that the connected shearlet group and Toeplitz shearlet group have an isomorphic subgroup within the symplectic group.
\subsection{Shearlet and Toeplitz shearlet groups}\label{subsec11}
For fixed $\gamma\in\RR$ 
(usually $0<\gamma < 1$, e.g., $\gamma = \frac{1}{d}$, to ensure directional selectivity)
and $a \in \RR^*:= \RR \setminus \{ 0 \}$
we introduce the \emph{dilation matrices} 
\begin{equation}\label{eq:dilationAndShearMatrices}
  A_{a,\gamma}
:=
  \begin{pmatrix}
    a & 0 \\
    0 & \sgn(a)\lvert a\rvert^\gamma I_{d-1}
  \end{pmatrix}
\qquad \text{and} \qquad
	A_a :=  a I_d
\end{equation}
and for $s=(s_1,\ldots,s_{d-1})^\tT\in\RR^d$ the \emph{shear} and \emph{Toeplitz shear matrices}
\begin{equation}\label{eq:shearmatrices}                                                                                                     
  S_s
:=
  \begin{pmatrix}
    1 & s^\tT \\
    0 & I_{d-1}
  \end{pmatrix}
\qquad \text{and} \qquad
  T_{s} 
:= 
  \begin{pmatrix} 
    1      & s_1   & s_2    & \ldots & s_{d-1} \\
    0      & 1     & s_1    & s_2    & \vdots  \\
    \vdots &\ddots & \ddots & \ddots & \vdots  \\
    \vdots &       & \ddots & 1      & s_1     \\
    0      & \dots & \ldots & 0      & 1
 \end{pmatrix}.
\end{equation}
Note that the product of two upper triangular Toeplitz matrices $T_s$ and $T_{s'}$ is again an upper triangular Toeplitz matrix
$T_{s \msharp s'}$ with
\begin{equation*}
  (s\msharp s')_i 
:= 
  s_i+s'_i + \sum_{j+k=i}s'_j s_k, \quad
  i=1,\ldots, d-1.
\end{equation*}
The shearlet and Toeplitz shearlet groups are defined as follows:
\begin{itemize}
 \item The (full) \emph{shearlet group} $\SS$ is the set $\RR^* \times \RR^{d-1} \times  (\RR \times \RR^{d-1})$
with the group operation
\begin{equation*}
  (a,s,t)\circ_{\SS}(a',s',t')
:=
  (aa', s + \lvert a\rvert^{1-\gamma}s', t + S_s A_{a,\gamma} t').
\end{equation*}
Using the notation $t=(t_1,\tilde{t})^\tT\in\RR^d$ the group law can be rewritten as
\begin{equation}\label{eq:shearlet_grouplaw_modified}
  (a,s,t_1,\tilde{t})\circ_{\SS}(a',s',t_1',\tilde{t}')
=
  (aa', s + \lvert a\rvert^{1-\gamma}s',
  t_1 + at_1' + \sgn(a)\lvert a\rvert^{\gamma}s^\tT \tilde{t}',
  \tilde{t} + \sgn(a)\lvert a\rvert^\gamma \tilde{t}').
\end{equation}
\item The \emph{connected shearlet group} $\SS^{+}$ is the set
$\RR^+\times \RR^{d-1} \times (\RR \times \RR^{d-1})$ with group law
\eqref{eq:shearlet_grouplaw_modified} reduced to positive $a = \lvert a\rvert =  \sgn(a)\lvert a\rvert$.
\item The (full) \emph{Toeplitz shearlet group} $\SS_T$ is the set $\RR^* \times \RR^{d-1} \times  \RR^{d}$
with the group operation
\begin{equation} \label{law_toeplitz}
  (a,s,t) \circ_{\SS^{+}_T} (a',s',t') 
= 
  (aa', s\msharp s', t + A_a T_{s}  t').
\end{equation}
\item The \emph{connected Toeplitz shearlet group} $\SS_T^{+}$ is the set
$\RR^+\times \RR^{d-1} \times \RR^{d}$ with group law \eqref{law_toeplitz}
restricted to positive $a$.
\end{itemize}
The four groups are locally compact groups, where the left and right Haar measures
of the shearlet groups are given by
\begin{equation*}
  d\mu_{\SS,l}(a,s,t)
:=
  \frac{1}{\lvert a \rvert^{d+1}}\, da\, ds\, dt
\quad\text{and}\quad
  d\mu_{\SS,r}(a,s,t)
:=
  \frac{1}{\lvert a \rvert}\, da\, ds\, dt.
\end{equation*}
see \cite{DKMSST08,DST10} 
and of the Toeplitz shearlet groups by 
\begin{equation*}
  d \mu_{\SS_T,l} (a,s,t) 
= 
  \frac{1}{\lvert a\rvert^{d+1}} \, da \, ds  \, dt
\quad\text{and}\quad
  d \mu_{\SS_T,r} (a,s,t) 
= 
  \frac{1}{\lvert a\rvert} \, da \, ds \, dt,
\end{equation*}
see \cite{DHT12,DT10}
with restriction to positive dilations $a$ for the connected groups.
The connected (Toeplitz) shearlet group is a subgroup of the the full (Toeplitz) shearlet group.
%
\subsection{Relation to Heisenberg groups}\label{subsec12}
%
The Heisenberg group and its polarized version are defined as follows:
\begin{itemize}
 \item 
The \emph{Heisenberg group} $\HH$  is the set 
  $\RR^{d-1} \times \RR \times \RR^{d-1}$ 
  endowed with the group operation
  \begin{equation*}
    (p,\tau,q)\circ_\HH (p',\tau',q')
  :=
    \bigl(p + p',\tau + \tau' + \tfrac{1}{2}(p^\tT q' - q^\tT p'), q+q'\bigr).
  \end{equation*}
\item  The \emph{polarized Heisenberg group} $\HH^{\operatorname{pol}}$
is the same set 
  $\RR^{d-1} \times \RR \times \RR^{d-1}$ 
  but with the group operation
  \begin{equation}\label{eq:group_law_polarized_heisenberg}
    (p,\tau,q)\circ_{\HH^{\operatorname{pol}}} (p',\tau',q') 
  := 
    (p + p', \tau + \tau' +  p^\tT q',q + q').
  \end{equation}
\end{itemize}
These Heisenberg groups are isomorphic with isomorphism given by 
\begin{equation*}
  \phi\colon \HH \to \HH^{\operatorname{pol}}, 
\quad
  (p,\tau,q) \mapsto (p,\tau + \tfrac{1}{2} p^\tT q,q).
\end{equation*}
This is why we usually write \emph{the} Heisenberg group.
If we set $a = a' = 1$ in \eqref{eq:shearlet_grouplaw_modified} we obtain
\begin{equation*}
  (1,s,t_1,\tilde{t})\circ_\SS(1,s',t'_1,\tilde{t}')
=
  (1, s + s', t_1 + t_1' + s^\tT \tilde{t}', \tilde{t} + \tilde{t}')
\end{equation*}
which looks very similar to the group law of the Heisenberg group in \eqref{eq:group_law_polarized_heisenberg}. 
We will show that the shearlet group is isomorphic to an extended 
Heisenberg group which is equipped with a dilation. 
For the general concept of group extensions we refer to \cite{Kna02}.
We briefly recall the notion of a semi-direct product.
Given 
a group $H$ and a group $G$ acting on $H$ by automorphisms,
i.e., a smooth map 
$\delta\colon G\times H\to H$ is defined such that
$\delta(g,\,\cdot\,)$ is an automorphism of $H$, we can extend $H$ by $G$
by forming the semi-direct product $H\rtimes G$.
The multiplication and inversion are determined by
\begin{equation*}
  (h,g)(h',g')
=
  (h \circ_H \delta(g,h'),g \circ_G g')
\quad\text{and}\quad
  (h,g)^{-1}
=
  (\delta(g^{-1},h^{-1}),g^{-1}).
\end{equation*}
The extended Heisenberg group is the semi-direct product of the Heisenberg group $\HH$ and $\RR^*$,
where $\RR^*$ acts on $\HH$ via the automorphism 
\begin{equation*}
  \delta_a^\gamma(p,\tau,q)
:=
  (\lvert a\rvert^{1-\gamma}p,a\tau, \sgn(a)\lvert a\rvert^\gamma q), \quad \gamma > 0,
\end{equation*}
for details see \cite{Hae14}.
In other words,
\begin{itemize}
\item
the \emph{extended Heisenberg group} is defined by
  $\HH_e := \HH \rtimes\RR^* $ with the group operation 
	\begin{multline*}
    (p,\tau,q,a)\circ_{\HH_e} (p',\tau',q',a') :=
  \\
    \bigl(      
      p + \lvert a\rvert^{1-\gamma} p',
      \tau + a \tau' + \tfrac{1}{2}\bigl(\sgn(a)\lvert a\rvert^\gamma p^\tT q' - \lvert a\rvert^{1-\gamma} q^\tT p'\bigr),
      q + \sgn(a)\lvert a\rvert^\gamma q',
      aa' 
    \bigr)
  \end{multline*}
	\item  and the \emph{extended polarized Heisenberg group} by
  $\HH^{\operatorname{pol}}_e:= \HH^{\operatorname{pol}} \rtimes\RR^*$
  with group operation
  \begin{equation}\label{eq:heisenberg_grouplaw}
   ( p,\tau,q,a)\circ_{\HH_e^{\operatorname{pol}}} (p',\tau',q',a')
  :=
    \bigl(
      p + \lvert a\rvert^{1-\gamma} p',
      \tau + a\tau' +  \sgn(a)\lvert a\rvert^\gamma p^\tT q',
      q + \sgn(a)\lvert a\rvert^\gamma q', aa'
    \bigr).
    \end{equation}
\end{itemize}
These groups are again isomorphic with
\begin{equation*}
  \phi_e\colon \HH_e \to \HH_e^{\operatorname{pol}},
\quad 
  (p,\tau,q,a) \mapsto (p,\tau + \tfrac{1}{2} p^\tT q,q,a).
\end{equation*}
Using only positive dilations we obtain the connected
versions of the extended (polarized) Heisenberg group, 
whose composition laws we do not write explicitly.
%
For $\gamma=\frac{1}{2}$ the dilation is symmetric in $p$ and $q$. 
Comparing the definition of $\SS$ and $\HH^{\operatorname{pol}}_e$ 
we see that both groups coincide up to a permutation of
the variables.
The same holds true for the connected group variants.
Taking further the isomorphism between the Heisenberg groups and its
polarized versions into account we can summarize:
\begin{lemma} {\rm (Relation between extended Heisenberg groups and shearlet groups) } \label{lem:HS}\\
The following relations hold true:
\begin{equation*}
  \HH_e \cong \HH_e^{\operatorname{pol}} = \SS
\qquad \text{and} \qquad
  \HH_e^+ \cong  \HH_e^{\operatorname{pol},+}  = \SS^+.
\end{equation*}
\end{lemma}
%
\subsection{Relation to subgroups of the symplectic group}\label{subsec13}
Let $\GL(d,\RR)$ denote the general linear group of real, invertible $d \times d$ matrices.
The \emph{symplectic group} $\Sp(d,\RR)$  is the group of all matrices $B \in \GL(2d,\RR)$ fulfilling 
$B^\tT J B = J$ for 
$
  J 
= 
  \left(\begin{smallmatrix} 0 & I_d \\ - I_d & 0 \end{smallmatrix}\right)
$, i.e.,
\begin{equation*}
  \Sp(d,\RR)
:=
  \{
    B\in\RR^{2d\times 2d}
  :
    B^\tT J B = J
  \}.
\end{equation*}
Let $H$ be a closed subgroup of $\GL(d,\RR)$ and $\Sigma$ an additive subspace of 
the symmetric matrices $\Sym(d,\RR)$ that is invariant under the $H$-action given by 
$M^{-\tT}\sigma M^{-1}\in\Sigma$ for all $M\in H$ and $\sigma\in\Sigma$. 
Then we know by \cite[Example 3]{DD11} that the 
  semi-direct product 
\begin{equation} \label{general_metaplectic}
   \Sigma\rtimes H
  :=
    \left\{
    \begin{pmatrix}
      M & 0 \\ \sigma M & M^{-\tT}
    \end{pmatrix}
    :
    M \in  H, \sigma\in\Sigma
    \right\}
\end{equation}
is a subgroup of $\Sp(d,\RR)$. 
We are interested in two special groups of the form \eqref{general_metaplectic}.
The first one is the group $\TDS(d)$ of 
translations, dilations and shears defined by
\begin{equation*}
    \TDS(d)
  :=
    \left\{
      \begin{pmatrix}
        M(s,a) & 0 \\ \sigma(t)M(s,a) & M(s,a)^{-\tT}
      \end{pmatrix}
      :
      a\in\RR^+, \, s\in\RR^{d-1}, \, t\in\RR^d
    \right\},
  \end{equation*}
where
\begin{equation} \label{tilde_a_s}
    \widetilde{A}_{a,\gamma}
  :=
    \begin{pmatrix} 
      a^{-\frac{1}{2}} & 0 \\ 
      0                & a^{\frac{1}{2}-\gamma}I_{d-1}
    \end{pmatrix}
    ,\quad
    \widetilde{S}_s
  :=
    \begin{pmatrix} 
       1 & 0 \\ 
      -s & I_{d-1} 
    \end{pmatrix}
  ,\quad
    M(s,a) 
  := 
    \widetilde{S}_s\widetilde{A}_{a,\gamma},
\end{equation}
and  $\sigma(t)$ belongs to the subspace of the symmetric matrices
\begin{equation} \label{def:spec_sigma}
  \{ 
     \sigma(t) 
  = 
    \sigma(t_1,\tilde{t})
  :=
    \begin{pmatrix}
      t_1                  & \frac{1}{2}\tilde{t}^\tT \\ 
      \frac{1}{2}\tilde{t} & 0_{d-1,d-1}
    \end{pmatrix}: t \in \RR^d
  \}.
\end{equation}
Straightforward computation shows that the subspace \eqref{def:spec_sigma} 
is indeed invariant under $H$-action with matrices $M(s,a)$.
 
The relation between  this subgroup $\TDS(d)$ of the symplectic group and the connected shearlet group $\SS^{+}$ is 
stated in the following lemma. 
For a proof  we refer to \cite{Hae14,DD11}.
%
\begin{lemma} {\rm (Relation between $\SS^+$ and  $\TDS(d)$)} \label{lemma:isomorphism_Splus_TDS}\\
 The groups $\SS^{+}$ and $\TDS(d)$ are isomorphic and the isomorphism $\kappa^{+}$ is given by 
\begin{equation} \label{eq:definition_g_+} 
\kappa^{+} \colon \SS^{+} \rightarrow \TDS(d),\qquad
  (a,s,t_1,\tilde{t})
\mapsto
  \begin{pmatrix}
    M(s,a) & 0 \\
    \sigma(t_1,\tilde{t})M(s,a) & M(s,a)^{-\tT }
  \end{pmatrix}.
\end{equation}
\end{lemma}
%
The second interesting group of the form \eqref{general_metaplectic}  is
the group 
of translations, dilations and Toeplitz shears  given by 
\begin{equation*}
  \TDS_T(d) 
= 
  \left\{
    \begin{pmatrix}
      a^{-1/2} T_s^{-\tT}                        & 0           \\
      a^{-1/2}\sigma(t) T_s^{-\tT} & a^{1/2} T_s
    \end{pmatrix}
  :
    a\in\RR^+, s\in\RR^{d-1}, t \in\RR^d
  \right\}.
\end{equation*}
Clearly, the subspace \eqref{def:spec_sigma} is invariant under the $H$-action with matrices $a^{-1/2} T_s^{-1}$.
This group is related to the connected Toeplitz shearlet group as follows:
%
\begin{lemma} {\rm (Relation between $\SS^+_T$ and  $\TDS_T(d)$)} \label{lemma:isomorphism_Splus_TDS_toep}\\
 The groups $\SS^{+}_T$ and $\TDS_T(d)$ are isomorphic and the isomorphism $\kappa_T^{+}$ is given by 
\begin{equation} \label{eq:definition_g_+_T} 
\kappa_{T}^{+}\colon\SS_T^+\to \TDS_T(d),  \quad
  (a,s,t_1,\tilde{t})
\mapsto
  \begin{pmatrix}
    a^{-1/2} T_s^{-\tT}                        & 0           \\
    a^{-1/2}\sigma(t_1,\tilde{t}) T_s^{-\tT} & a^{1/2} T_s
  \end{pmatrix}.
\end{equation}
\end{lemma}
%
\begin{proof}
From the group law in the Toeplitz shearlet group we know that 
\begin{equation*}
  (a,s,t) \circ_{\SS^{+}_T} (a',s',t') 
= 
  (aa', s\msharp s', t + a T_{s}  t'),
\end{equation*}
so that we have to show 
\begin{equation*}
  \kappa_T^{+}(a,s,t) \circ \kappa_T^{+}(a',s',t')
= 
  \kappa_T^{+}(aa', s\msharp s', t + a T_{s}  t').
\end{equation*}
The right hand side can be rewritten as
\begin{equation*}
  \kappa_T^{+}(aa', s\msharp s', t + a T_{s}  t')
=
  \begin{pmatrix}
    (aa')^{-\frac{1}{2}} T_{s\msharp s'}^{-\tT} & 0 \\
    (aa')^{-\frac{1}{2}} 
    \sigma\bigl(t_1 + a (t_1' + s^\tT \tilde{t}'),\tilde{t} + a T_{[s]}\tilde{t}'\bigr)
    T_{s\msharp s'}^{-\tT} 
        & (aa')^\frac{1}{2}T_{s\msharp s'}
  \end{pmatrix}
\end{equation*}
where $[s] = (s_i)_{i=1}^{d-2}\in\RR^{d-2}$. 
For the left hand side we have
\begin{align*}
&
  \begin{pmatrix}
    a^{-1/2} T_s^{-\tT}                        & 0           \\
    a^{-1/2}\sigma(t_1,\tilde{t}) T_s^{-\tT} & a^{1/2} T_s
  \end{pmatrix}
  \begin{pmatrix}
    (a')^{-1/2} T_{s'}^{-\tT}                        & 0           \\
    (a')^{-1/2}\sigma(t_1',\tilde{t}') T_{s'}^{-\tT} & (a')^{1/2} T_{s'}
  \end{pmatrix}
\\
&=
  \begin{pmatrix}
    (aa')^{-1/2} T_{s\msharp s'}^{-\tT}                        & 0           \\
    (aa')^{-1/2} \big(\sigma(t_1,\tilde{t}) + a T_s \sigma(t_1',\tilde{t}') T_s^\tT \big) T_{s\msharp s'}^{-\tT} & (aa')^{1/2} T_{s\msharp s'}
  \end{pmatrix}
\end{align*}
consequently, it is sufficient to show 
\begin{equation}\label{eq:toeplitz_conjugation}
    \begin{pmatrix}
    t_1' + s^\tT\tilde{t'}                      & \frac{1}{2} (T_{[ s]} \tilde{t}')^\tT \\
    \frac{1}{2} T_{[ s ]}\tilde{t}'  & 0
  \end{pmatrix}
=
  T_s
  \begin{pmatrix}
    t_1'                      & \frac{1}{2} (\tilde{t}')^\tT \\
    \frac{1}{2}\tilde{t}' & 0
  \end{pmatrix}
  T_s^\tT.
\end{equation}
The right hand side of \eqref{eq:toeplitz_conjugation} is
\begin{align*}
&
  \begin{pmatrix}
    1 & s_1    & \cdots & s_{d-1} \\
      & \ddots &        & \vdots  \\
      &        & \ddots & s_1     \\
      &        &        & 1
  \end{pmatrix}
  \begin{pmatrix}
    t_1'                   & \frac{1}{2} (\tilde{t}')^\tT \\
    \frac{1}{2}\tilde{t}' & 0
  \end{pmatrix}
  \begin{pmatrix}
    1       &        &        &  \\
    s_1     & \ddots &        &  \\
    \vdots  &        & \ddots &  \\
    s_{d-1} & \cdots & s_1    & 1
  \end{pmatrix}
\\
&=
  \begin{pmatrix}
    t_1' + \frac{1}{2}s^\tT\tilde{t}' & \frac{1}{2} (\tilde{t}')^\tT \\
    \frac{1}{2} T_{[ s ]}\tilde{t}'   & 0
  \end{pmatrix}
  \begin{pmatrix}
    1       &        &        &      \\
    s_1     & \ddots &        &      \\
    \vdots  &        & \ddots &      \\
    s_{d-1} & \cdots & s_1       & 1
  \end{pmatrix}
=
  \begin{pmatrix}
    t_1' + \frac{1}{2}s^\tT\tilde{t}' + \frac{1}{2}s^\tT\tilde{t}'        & \frac{1}{2}(T_{[ s ]}\tilde{t}')^\tT \\
    \frac{1}{2} T_{[ s ]}\tilde{t}' & 0
  \end{pmatrix}
\end{align*}
which coincides with the left hand side of \eqref{eq:toeplitz_conjugation} and we are done.
\end{proof}
%
\section{Embedding of the full shearlet group} \label{sec3}
%
In the following we want to prove that it is not possible 
to embed the full shearlet group into the symplectic group
$\Sp(2,\RR)$ or one of its covers.
To show this result we pursuit the following path:
a) establish a necessary property for those continuous, injective group homomorphisms of $\SS^+$ to $\Sp(d,\RR)$
which can be extended to $\SS$;
b) show that this property is not fulfilled by the special homomorphism $\kappa^+$ defined in \eqref{eq:definition_g_+}
nor by its conjugation or concatenations with isomorphisms of $\SS^+$;
c) prove that in dimension $d=2$ actually any continuous, injective group homomorphism is given
up to conjugation or concatenations with isomorphisms of $\SS^+$ by the map $\kappa^+$.

In the following $\SS$ is regarded as the semi-direct
product of its closed normal subgroup $\SS^+$ and its finite subgroup 
$\{(\pm 1,0,0,0)\}\simeq \ZZ_2$. 
Indeed, 
\begin{equation*}
  (-1,0,0,0) 
\circ_{\SS}
  (a,s,t_1,\tilde{t}) 
\circ_{\SS}
  (-1,0,0,0) 
= 
  (a,s,-t_1,-\tilde{t})
\end{equation*}
and, clearly,
\begin{equation*}
  \SS^+\circ_{\SS}\ZZ_2
=
  \SS 
\quad\text{and}\quad
  \SS^+\cap \ZZ_2
=
  \{(1,0,0,0)\}.
\end{equation*}
With slight abuse of notation, we write an element of $\SS$ as a pair $(x,\varepsilon)$ where
$x = (a,s,t_1,\tilde{t})\in\SS^{+}$ and $\varepsilon \in \ZZ_2$.
The group operation in $\SS$ becomes
\begin{equation*}
  (x,\varepsilon)\circ_{\SS} (x',\varepsilon')
=
  (x \circ_\SS R_\varepsilon x',\varepsilon\varepsilon'),
\end{equation*}
where $R_\eps$ is the group isomorphism of $\SS^+$ given by
\begin{equation}\label{eq:2}
  R_\varepsilon x 
= 
  R_\varepsilon (a,s,t_1,\tilde{t}) 
= 
  (a,s,\varepsilon t_1, \varepsilon \tilde{t}).
\end{equation}
Let $e := (1,0,0,0)$ denote the identity of $\SS^+$.
Then, in particular,
\begin{equation}
  \label{eq:5}
  (e,-1)\circ_{\SS} (x,1) 
= 
  (R_{-1} x,1) \circ_{\SS} (e,-1) 
\end{equation}
and $(e,1)$ is the identity of $\SS$.
%
\begin{lemma}\label{lemma:injective_extension}
  An injective group homomorphism $g^+\colon\SS^{+} \to \Sp(d,\RR)$ extends to a group homomorphism 
  $g \colon \SS\to\Sp(d,\RR)$ if and only if there exists $A\in\Sp(d,\RR)$ such that
  \begin{gather}\label{eq:Aconjugation}
    A^2 =  I_{2d}
  \quad \text{and} \quad
     A g^+(x) = g^+(R_{-1}x) A
  \end{gather}
	for all $x\in\SS^{+}$.
  Under these assumptions we have for all $x\in\SS^{+}$ that
  \begin{equation} \label{eigs}
    g(x,1)
  =
    g^+(x)
  \quad\text{and}\quad
   g(x,-1)
  =
    g^+(x)A.
  \end{equation}
  The extended group homomorphism $g$ is injective.
\end{lemma}
%
Note that by injectivity of $g^+$ we have $A \not = I$ for any $A$ fulfilling \eqref{eq:Aconjugation}.
\begin{proof}
$\Rightarrow$:
  Assume that $g^+\colon\SS^{+} \to \Sp(d,\RR)$ extends to a homomorphism $g\colon\SS \to \Sp(d,\RR)$.
	Then we have in particular $g(x,1) = g^+(x)$ for all $x \in \SS^+$.
	We show that $A := g(e,-1)$ fulfills \eqref{eq:Aconjugation}.
	Since 
	$(e,-1) \circ_{\SS} (e,-1) = (e \circ_{\SS^+} R_{-1} e,1) = (e,1)$
	and $g$ is a homomorphism we obtain $A^2 = I_{2d}$.
	By 
	\begin{equation*} 
	(e,-1) \circ_{\SS} (x,1) = (e \circ_{\SS^+} R_{-1} x, -1) = (R_{-1} x,1) \circ_{\SS} (e,-1)
	\end{equation*}
	and the fact that $g$ is a homomorphism we get the second equality in \eqref{eq:Aconjugation}.
	Finally, we conclude since 
	$(x,-1) \circ_{\SS} (e,-1) = (x \circ_{\SS^+} R_{-1} e,1) = (x,1)$,
	$g$ is a homomorphism and $A^2 = I_{2d}$ that $g(x,-1) = g^+(x) A$.
		\\	
$\Leftarrow$:
  Conversely, assume that there exists $A \in \Sp(d,\RR)$ satisfying \eqref{eq:Aconjugation}. Set 
  \begin{equation*}
    g(x,\varepsilon)
  :=
    \begin{cases}
      g^+(x)  &\text{ for }\varepsilon = 1, \\
      g^+(x)A &\text{ for }\varepsilon = -1.
    \end{cases}
  \end{equation*}
	Then \eqref{eigs} is fulfilled by definition.
  Direct computation shows that $g$ is a group homomorphism from $\SS$ into $\Sp(d,\RR)$. 
	\\
  It remains to prove the injectivity of $g$.
  By \eqref{eigs} and the invertibility of $A$ we see that
  $g(x,\varepsilon) = g(x',\varepsilon)$ implies $g^+(x) = g^+(x')$ and since $g^+$ is injective 
  further $x=x'$. 
  If
  $g(x',1) = g(x,-1)$
  we obtain
  $g^+(x') = g^+(x)A$
  and since $g^+$ is a homomorphism that
  $g^+(x^{-1}x') = A$.
  Set $y := x^{-1}x'\in\SS^{+}$. 
  Since $g^+$ is a homomorphism we conclude $g^+(y^2) = (g^+(y))^2 = A^2 = I_{2d}$ and
  with the injectivity of $g^+$ that $y^2 = e$. 
  But this is only possible if  $y = e$ and consequently $g^+(y) = I_{2d} = A$
  which is a contradiction. Hence $g$ is injective.
\end{proof}
We recall that a covering group of $\Sp(d,\RR)$ is
a connected Lie group $G$ with a surjective 
continuous group homomorphism $p\colon G\to\Sp(d,\RR)$
whose kernel is discrete.
\begin{lemma} \label{lemma:covering}
  Let $(G,p)$ be a covering group of $\Sp(d,\RR)$ with covering homomorphism $p$ 
  and an injective 
  continuous group homomorphism $i\colon\SS\to G$. 
  Then $p\circ i$ is an injective group homomorphism of $\SS$ into $\Sp(d,\RR)$.
\end{lemma}
\begin{proof}
By definition it is clear that $p\circ i$ is a homomorphism.
Next we have that $p\circ i$ restricted to $\SS^{+}$ is injective by the following argument.
Recall that  a continuous group homomorphism of a Lie group is always smooth.
Since $\SS^{+}$ is a connected, simply connected Lie group it is enough to prove that its tangent map at the identity $(p\circ i)_*\big\vert_e$ is injective. Observe that 
$
  (p\circ i)_*\big\vert_e
=
  p_*\big\vert_{e_G}\,
  i_*\big\vert_e.
$
Since $i$ is injective, the same holds true for $i_*\big\vert_e$ 
and since $p$ is a covering homomorphism $p_*\big\vert_{e_G}$ is injective. Thus their concatenation is injective.
Now applying Lemma~\ref{lemma:injective_extension} with $g^+:= p\circ i|_{\SS^+}$ yields the assertion.
\end{proof}
We have shown that a special injective homomorphism from $\SS^+$ into $\Sp(2,\RR)$ is given by 
$g^+:= \kappa^+$ defined in \eqref{eq:definition_g_+}.
%
\begin{lemma}\label{lemma:no_A_for_Filippo_map}
  For $\kappa^+\colon \SS^+ \rightarrow \Sp(d,\RR)$ defined by \eqref{eq:definition_g_+} 
  there does not exist $A \in \Sp(d,\RR)$ satisfying \eqref{eq:Aconjugation}.
  The same holds true for 
  \begin{enumerate}
	\item[\textrm{(i)}]
      any conjugation of $\kappa^+$, i.e., for any map
      $\kappa^+_B\colon \SS^+\to\Sp(d,\RR)$ with
      $
        \kappa^+_B(x) 
      := 
        B \, \kappa^{+}(x) B^{-1}
      $,
      $B\in\Sp(d,\RR)$,
    \item[\textrm{(ii)}] 
      any map $\kappa^+_\varphi\colon \SS^+\to\Sp(d,\RR)$ of the form 
      $\kappa^+_\varphi(x) := \kappa^{+}(\varphi(x))$, 
      where $\varphi$ is a group automorphism of $\ \SS^+$ such that
      \begin{equation} \label{isoprop}
        \varphi(R_{-1} x)
      =
        R_{-1}\varphi(x)
      \end{equation}
      for all $x\in \SS^+$.
  \end{enumerate}
\end{lemma}
\begin{proof}
1. Assume that there exists 
$
  A 
:=
  \left(
    \begin{smallmatrix} 
      \alpha & \beta \\ \gamma & \delta 
    \end{smallmatrix}
  \right)
\in 
  \Sp(d,\RR)
$ 
satisfying \eqref{eq:Aconjugation} for $\kappa^+$.

Since $A$ is symplectic it holds $A^\tT J A = J$ and since $A^{-1} = A$ further $A^\tT J = J A$. 
Hence
  \begin{equation*}
    \begin{pmatrix}
      \alpha^\tT & \gamma^\tT \\ \beta^\tT & \delta^\tT
    \end{pmatrix}
    \begin{pmatrix}
      0 & I_d \\ -I_d & 0
    \end{pmatrix}
  =
    \begin{pmatrix}
      -\gamma^\tT & \alpha^\tT \\ -\delta^\tT &\beta^\tT
    \end{pmatrix}
  =
    \begin{pmatrix}
      0 & I_d \\ -I_d & 0
    \end{pmatrix}
    \begin{pmatrix}
      \alpha & \beta \\ \gamma & \delta
    \end{pmatrix}
  =
    \begin{pmatrix}
      \gamma & \delta \\ -\alpha & -\beta
    \end{pmatrix}
  \end{equation*}
  which implies  
  $\beta = -\beta^\tT$, $\gamma = -\gamma^\tT$ and $\delta = \alpha^\tT$.
  Thus, 
  $A = \left(\begin{smallmatrix} \alpha & \beta \\ \gamma & \alpha^\tT \end{smallmatrix}\right)$ 
  with skew-symmetric $\beta,\gamma$.
  In particular, $\beta$ and $\gamma$ have zero diagonal elements.
  The second condition in \eqref{eq:Aconjugation} with 
  $x := (1,0,t_1,\tilde{t})$ results by definition of $\kappa^+$ in 
  \begin{equation*}
    A
    \begin{pmatrix}
      I_d & 0 \\ \sigma(t_1,\tilde{t}) & I_d
    \end{pmatrix}
  =
    \begin{pmatrix}
      I_d & 0 \\ -\sigma(t_1,\tilde{t}) & I_d
    \end{pmatrix}  
    A
  \end{equation*}
  and straightforward computation shows that this implies
  \begin{align} \label{eins}
    \beta \sigma(t_1,\tilde t) 
  &= 
    0 
  = 
    \sigma(t_1,\tilde t) \beta,
  \\
    \alpha^\tT \sigma(t_1,\tilde t) 
  &= 
    - \sigma(t_1,\tilde t) \alpha \label{zwei}
  \end{align}
  for all $t \in \RR^d$.
  Since $\beta = -\beta^\tT$, it has the form
  $
    \beta
  =
    \left(\begin{smallmatrix} 
      0 & -u^\tT \\ u & M
    \end{smallmatrix}\right)
  $
  with $u\in\RR^{d-1}$ and $M^\tT = -M\in\RR^{d-1 \times d-1}$.
  Choosing $t:=(1,0,\ldots,0)$ in \eqref{eins} we see immediately 
  by definition \eqref{def:spec_sigma} of $\sigma$ that $u = 0$.
  Let
  $
    \alpha
  =
    \left(\begin{smallmatrix}
      a & w^\tT \\
      v & N
    \end{smallmatrix}\right)
  $
  with $v,w\in\RR^d$ and $N\in\RR^{d-1 \times d-1}$. 
  Evaluating \eqref{zwei} for $t:=(1,0,\ldots,0)$ we obtain
  $a=0$ and $w = 0$. 
  In summary, the matrix $A$ is of the form
  \begin{equation*}
    A
  =
    \left(
    \begin{array}{cc}
      \begin{array}{cc}
        0 & \mathbf{0}^\tT \\  
       v & N              
      \end{array}
      &
      \begin{array}{cc}
        0 & \mathbf{0}^\tT \\
        \mathbf{0} & M 
      \end{array}
      \\
      \gamma
      &
      \begin{array}{cc}
        0 & v^\tT \\
        \mathbf{0} & N^\tT
    \end{array}
    \end{array}
    \right).
  \end{equation*}
  Evidently, $A$ is not invertible because it has a zero column.

  2. (i) Assume that there exists $A \in \Sp(d,\RR)$ satisfying \eqref{eq:Aconjugation} for some conjugation map $\kappa^+_B$.
  But then $\tilde A := B^{-1} A B \in  \Sp(d,\RR)$ fulfills \eqref{eq:Aconjugation} for $\kappa^+$. 
  This contradicts the first part of the
  proof. 

  (ii) Finally, assume there exists $A \in \Sp(d,\RR)$ such that
  $A\kappa^+_\varphi(y)A^{-1}=\kappa^+_\varphi(R_{-1} y)$ for all $y\in \SS^+$. 
  With $y=\varphi^{-1}(x)$ we get $A \kappa^+(y)A^{-1}= \kappa^+(R_{-1} y)$, 
  which is again a contradiction. 
\end{proof}
Lemma~\ref{lemma:injective_extension} and Lemma~\ref{lemma:no_A_for_Filippo_map} imply that the special
group homomorphism  
$\kappa^+\colon \SS^{+} \rightarrow  \Sp(d,\RR)$ in \eqref{eq:definition_g_+} as well as
its conjugations or concatenations with group automorphisms of $\SS^{+}$ satisfying \eqref{isoprop}
cannot be extended to a group homomorphism of $\SS$ to $\Sp(d,\RR)$.
For $d=2$ we have the sharper result that this holds true for \emph{all} 
injective continuous group homomorphisms $g^+\colon \SS^+ \rightarrow \Sp(2,\RR)$.
The following theorem will be proved in the next section.
%
\begin{theorem} \label{kommt_noch}
Let $\gamma\in (0,1)\setminus \{\frac{1}{3}, \frac{2}{3}\}$.
For any injective continuous group homomorphism 
$g^+\colon \SS^+ \rightarrow \Sp(2,\RR)$ 
there exists $B\in\Sp(2,\RR)$ and a continuous group
isomorphism  $\varphi$ of $\SS^+$ satisfying \eqref{isoprop}
such that 
\begin{equation*}
  g^+(x)
=
  B \kappa^+(\varphi(x)) B^{-1}.
\end{equation*}
\end{theorem}
As immediate consequence of the theorem,
Lemma~\ref{lemma:injective_extension} and
Lemma~\ref{lemma:no_A_for_Filippo_map}
we obtain our main result
(for $\gamma=\frac{1}{3}$ and $\gamma=\frac{2}{3}$ see Remark~\ref{ein_drittel}).
%
\begin{theorem} \label{main_shear}
  Let $\gamma\in (0,1)$. 
  There does not exist an injective continuous homomorphism from $\SS$ into $\Sp(2,\RR)$ 
  and into any of its coverings.
\end{theorem}
%
We state the above result for  $\gamma\in (0,1)$  
since this is the range of interest in the applications. 
However, a simple inspection of the proof of
Theorem~\ref{main_shear} shows that 
Theorem~\ref{kommt_noch} holds true for any
$\gamma\in\RR\setminus \{0,1\}$. 

It is not clear if the theorem can be generalized to higher dimensions $d>2$.
However, we conjecture that the result holds true in arbitrary dimensions.
%
\section{Proof of Theorem~\ref{kommt_noch}} \label{sec4}


We start by examining the Lie algebra of the symplectic group in the next subsection and use the findings
for our embedding result in Subsection~\ref{subsec:main_result}.
%
\subsection{Root space decomposition and canonical forms} 
%
The Lie algebra $\sptwor$ of the symplectic group $\Sp (2,\RR)$
consists of the real $4\times 4$ matrices, called \emph{Hamiltonians}, which satisfy the equation 
$X^\tT J+JX=0$. It is  the $10$-dimensional Lie algebra
\begin{equation*}
  \sptwor
=
  \left\{
    \begin{pmatrix}
      M_{11} &  M_{12}    \\
      M_{21}  & - M_{11}^\tT
    \end{pmatrix}
  :
    M_{11}\in\RR^{2 \times 2},
   \;M_{12},M_{21}\in\Sym(2,\RR)
  \right\}.
\end{equation*}
\paragraph{Root space decomposition.}
To prove our main embedding result we need a representation 
of Hamiltonians with respect to a certain basis of $\sptwor$
which we provide next.
The maximally non compact \emph{Cartan subalgebra} of $\sptwor$ is given by
\begin{equation*}
  \ga
:=
  \Bigl\{
    H_{a,b}
  :=
    \begin{pmatrix}
      a & 0 &  0 &  0 \\
      0 & b &  0 &  0 \\
      0 & 0 & -a &  0 \\
      0 & 0 &  0 & -b
    \end{pmatrix}
  : 
    a,b\in\RR
  \Bigr\}
\end{equation*}
and has the natural basis $\{H_{1,0},H_{0,1} \}$.
We define the linear functionals $\alpha$ and $\beta$ on $\ga$ by
\begin{equation*}
  \alpha\left(H_{a,b}\right)
:=
  a-b,
\qquad
  \beta\left(H_{a,b}\right)
:=
  2b.
\end{equation*}
The functionals in
\begin{equation*}
  \triangle := \triangle^+ \cup \triangle^-, \quad 
  \triangle^+: = \left\{\alpha,\beta,\alpha+\beta,2\alpha+\beta\right\}, \quad
  \triangle^- := \{-\nu:\nu\in\triangle^+\}
\end{equation*}
form a so-called \emph{root system}. The root system is meaningful 
since for any non-zero functional $\nu$ not contained in $\triangle$
the vector space
\begin{equation*}
  \gg_\nu
:=
  \left\{
    X\in\sptwor
  :
    [H,X]=\nu(H)X
    \text{ for all }H\in\ga
  \right\}
\end{equation*}
is trivial. The \emph{root spaces}  $\gg_\nu$, $\nu \in \triangle$, are one-dimensional
and the linear space associated with the zero functional 
$\gg_0=\ga$  is two-dimensional. The four root vectors $X_\nu$ spanning the space
$\gg_\nu$, $\nu \in \triangle^+$, are
\begin{footnotesize}
\begin{equation*}
  X_{\alpha}
:=
  \begin{pmatrix}
    0 & 1 &  0 & 0 \\
    0 & 0 &  0 & 0 \\
    0 & 0 &  0 & 0 \\
    0 & 0 & -1 & 0
  \end{pmatrix},\, 
  X_{\beta}
:=
  \begin{pmatrix}
    0 & 0 & 0 & 0 \\
    0 & 0 & 0 & 1 \\
    0 & 0 & 0 & 0 \\
    0 & 0 & 0 & 0
  \end{pmatrix},\, 
  X_{\alpha+\beta}
:=
  \begin{pmatrix}
    0 & 0 & 0 & 1 \\
    0 & 0 & 1 & 0 \\
    0 & 0 & 0 & 0 \\
    0 & 0 & 0 & 0
  \end{pmatrix},\, 
X_{2\alpha+\beta}
:=
  \begin{pmatrix}
    0 & 0 & 2 & 0 \\
    0 & 0 & 0 & 0 \\
    0 & 0 & 0 & 0 \\
    0 & 0 & 0 & 0
  \end{pmatrix}.
\end{equation*}
\end{footnotesize}
The root vectors $X_{-\nu}$ spanning $\gg_{-\nu}$, $-\nu \in \triangle^-$, 
are given by the Cartan involution
\begin{equation*}
	X_{-\nu} := - X_{\nu}^\tT.
\end{equation*}
The Lie algebra $\sptwor$ has the following vector space direct sum decomposition, known as
\emph{root space decomposition}: 
\begin{equation*}
  \sptwor
=
  \ga+\sum_{\nu\in\triangle}\gg_\nu.
\end{equation*}
To show our embedding result we will label Hamiltonians with respect to the basis 
\begin{align} \label{basis}
  \mathcal{B}
&=
  \Bigl\{
    X_{\alpha}, X_{\beta}, X_{\alpha+\beta}, X_{2\alpha+\beta}, 
    X_{-\alpha}, X_{-\beta}, X_{-\alpha-\beta}, X_{-2\alpha-\beta}, H_{1,0},H_{0,1}
  \Bigr\} 
\\
&= 
  \{B_k:k=1,\ldots,10\}, \nonumber
\end{align}
where the enumeration is with respect to the above ordering of the elements.
The following table contains the commutator rules of the basis elements:
 \begin{tiny}
\begin{table}[htbp]
  \resizebox{\textwidth}{!}{
  \begin{tabular}{l|rrrr|rrrr|rr} \toprule
    $[\,\cdot\,,\,\cdot\,]$ & $X_\alpha$ & $X_\beta$ & $X_{\alpha+\beta}$ & $X_{2\alpha + \beta}$ & $X_{-\alpha}$ & $X_{-\beta}$ & $X_{-\alpha-\beta}$ & $X_{-2\alpha-\beta}$ & $H_{1,0}$ & $H_{0,1}$ \\\midrule
    $X_\alpha$ & $0$ & $X_{\alpha+\beta}$ & $X_{2\alpha + \beta}$ & $0$ & $H_{-1,1}$ & $0$ & $-2X_{-\beta}$ & $-2X_{-\alpha-\beta}$ & $-X_\alpha$ & $X_\alpha$ \\
    $X_\beta$ & $-X_{\alpha+\beta}$ & $0$ & $0$ & $0$ & $0$ & $-H_{0,1}$ & $X_{-\alpha}$ & $0$ & $0$ & $-2X_\beta$ \\
    $X_{\alpha+\beta}$ & $-X_{2\alpha + \beta}$ & $0$ & $0$ & $0$ & $2X_\beta$ & $-X_\alpha$ & $-H_{1,1}$ & $2X_{-\alpha}$ & $-X_{\alpha+\beta}$ & $-X_{\alpha+\beta}$ \\
    $X_{2\alpha + \beta}$ & $0$ & $0$ & $0$ & $0$ & $2X_{\alpha+\beta}$ & $0$ & $-2X_\alpha$ & $-H_{4,0}$ & $-2X_{2\alpha + \beta}$ & $0$ \\\midrule
    $X_{-\alpha}$ & $-H_{-1,1}$ & $0$ & $-2X_\beta$ & $-2X_{\alpha+\beta}$ & $0$ & $X_{-\alpha-\beta}$ & $X_{-2\alpha-\beta}$ & $0$ & $X_{-\alpha}$ & $-X_{-\alpha}$ \\
    $X_{-\beta}$ & $0$ & $H_{0,1}$ & $X_\alpha$ & $0$ & $-X_{-\alpha-\beta}$ & $0$ & $0$ & $0$ & $0$ & $2X_{-\beta}$ \\
    $X_{-\alpha-\beta}$ & $2X_{-\beta}$ & $-X_{-\alpha}$ & $H_{1,1}$ & $2X_\alpha$ & $-X_{-2\alpha-\beta}$ & $0$ & $0$ & $0$ & $X_{-\alpha-\beta}$ & $X_{-\alpha-\beta}$ \\
    $X_{-2\alpha-\beta}$ & $2X_{-\alpha-\beta}$ & $0$ & $-2X_{-\alpha}$ & $H_{4,0}$ & $0$ & $0$ & $0$ & $0$ & $2X_{-2\alpha-\beta}$ & $0$ \\\midrule
    $H_{1,0}$ & $X_\alpha$ & $0$ & $X_{\alpha+\beta}$ & $2X_{2\alpha + \beta}$ & $-X_{-\alpha}$ & $0$ & $-X_{-\alpha-\beta}$ & $-2X_{-2\alpha-\beta}$ & $0$ & $0$ \\
    $H_{0,1}$ & $-X_\alpha$ & $2X_\beta$ & $X_{\alpha+\beta}$ & $0$ & $X_{-\alpha}$ & $-2X_{-\beta}$ & $-X_{-\alpha-\beta}$ & $0$ & $0$ & $0$ \\\bottomrule
  \end{tabular}
  }
  \caption{Commutator relations $[B_i,B_j]$ for $B_i,B_j \in\cB$, $i,j=1,\ldots,10$.}
  \label{tab:commutator_relations}
\end{table}
\end{tiny}
\paragraph{Canonical normal forms.}
Next we give the complete list of canonical normal forms to which we can reduce 
real $4 \times 4$ Hamiltonians by means of real symplectic conjugations, i.e., by applying
$\ad(B) X := B X B ^{-1}$ with $B \in \Sp (2,\RR)$.
For arbitrary space dimensions and symplectic spaces over any field 
the result is due to Williamson \cite{Wil36}.
For real Hamiltonians the characterization
can be found in a synthetic form in \cite[Appendix 6]{Arn89} which we briefly recall for $\sptwor$
below.

The canonical normal forms are closely related to the Jordan normal forms of Hamiltonians.
The eigenvalues of Hamiltonians are of \emph{four types}, namely
(i) real pairs ($+a,-a$), $a > 0$, (ii) purely imaginary pairs ($+bi, -bi$), $b> 0$,
(iii) quadruples ($\pm a  \pm ib$), $a>0$, $b>0$, and (iv) zeros.
The Jordan blocks for the two members of a pair have the same structure, 
and there is an even number of blocks of odd order with zero eigenvalue.

In Arnol'd's book \cite{Arn89} the canonical normal forms are nicely determined by the help of a quadratic form (Hamiltonian function).
To this end, note that any $X \in \sptwor$ is related to a  symmetric matrix $A \in\Sym(4,\RR)$ by $J X= A$.
Now, $A\in\Sym(4,\RR)$ and hence $X = -JA$ is completely determined by the quadratic form
 \begin{equation*}
   H_A(x) :=
  \frac{1}{2} \langle Ax,x \rangle.
\end{equation*}
Using the notation $x:=(p_1,\ldots,p_k, q_1,\ldots,q_k)^\tT$, $k \in \{1,2\}$, 
the list of canonical normal forms for the irreducible cases and their relation to their Jordan normal forms read as follows (order as in \cite{Arn89}):
\begin{enumerate} [(A)]
  \item 
  If $X \in \mathfrak{sp}(1,\RR)$ has a pair of Jordan blocks of order one with real eigenvalues $\pm a$, $a \geq 0$, then it has the normal form
  \begin{equation*}
    H_A(p_1,q_1)
  =
    -ap_1q_1
  \quad \text{and} \quad
    -JA 
  = 
    \begin{pmatrix} 
      a & 0 \\ 0 & -a
	\end{pmatrix}.
  \end{equation*}
    \item 
  If $X \in \mathfrak{sp}(2,\RR)$ has Jordan blocks of order two with real eigenvalues $\pm a$, $a \geq 0$, then 
  \begin{equation*}
    H_A(p_1,p_2,q_1,q_2)
  =
    -a(p_1q_1+p_2q_2) + p_1q_2
  \quad \text{and} \quad
    -JA = D_2 
  = 
    \begin{pmatrix}
      a  & 0 & 0  & 0  \\
      -1 & a & 0  & 0  \\
      0  & 0 & -a & 1  \\
      0  & 0 & 0  & -a
    \end{pmatrix}.
  \end{equation*}
    \item 
  If $X \in \mathfrak{sp}(2,\RR)$ has a quadruple of Jordan blocks of order one with complex eigenvalues $\pm a \pm ib$, $a,b>0$, then      
	\begin{equation*}
    H_A(p_1,p_2,q_1,q_2)
  =
    -a(p_1q_1+p_2q_2)+b(p_1q_2-p_2q_1)
		\  \text{and} \ 
		-JA = D_3 = 
				\begin{pmatrix}
     a & b &  0 &  0 \\
    -b & a &  0 &  0 \\
     0 & 0 & -a &  b \\
     0 & 0 & -b & -a
  \end{pmatrix}.
  \end{equation*}
  \item 
  If $X \in \mathfrak{sp}(2,\RR)$ has a single Jordan block of order four with eigenvalue zero, then, for $\varepsilon = \pm 1$,
  \begin{equation*}
    H_A(p_1,p_2,q_1,q_2)
  =
    \frac{\eps}{2}(p_1^2-2q_1q_2)-p_1q_2
		\quad \text{and} \quad
		-JA = D_4 = \begin{pmatrix}
    0    & 0 & 0    & \eps \\
    1    & 0 & \eps & 0    \\
    \eps & 0 & 0    & -1   \\
    0    & 0 & 0    & 0
  \end{pmatrix}.
  \end{equation*}
  \item 
  If $X \in \mathfrak{sp}(1,\RR)$ has a pair of Jordan blocks of order one 
  with purely imaginary eigenvalues $\pm ib$, $b>0$, then,
  for $\varepsilon = \pm 1$,
  \begin{equation*}
    H_A(p_1,q_1)
  =
    -\frac{\eps}{2}(b^2p_1^2+q_1^2)
		\quad \text{and} \quad
		-JA = \begin{pmatrix} 0&\varepsilon\\-\varepsilon b^2&0\end{pmatrix}.
  \end{equation*}
  If $X \in \mathfrak{sp}(1,\RR)$ has a single Jordan block of order two with eigenvalue zero, then it has the canonical normal 
  $    H_A(p_1,q_1)   =
    -\frac{\eps}{2}q_1^2
	$,
	which coincides with the above form for $b=0$.
	\item 
  If $X \in \mathfrak{sp}(2,\RR)$ has a pair of Jordan blocks of order two 
  with purely imaginary eigenvalues $\pm ib$, $b>0$, then,
	for $\varepsilon = \pm 1$,
  \begin{equation*}
    H_A(p_1,p_2,q_1,q_2)
  =
    -\frac{\eps}{2}(\frac{1}{b^2}q_1^2+q_2^2)-b^2p_1q_2+p_2q_1
  \  \text{and} \ 
    -JA 
  = 
    D_7 
  = 
    \begin{pmatrix}
      0   & -1 & \frac{\eps}{b^2} &  0    \\
      b^2 &  0 &  0               &  \eps \\
      0   &  0 &  0               & -b^2  \\
      0   &  0 &  1               &  0
    \end{pmatrix}.
  \end{equation*}
\end{enumerate}
Combining the  irreducible $2 \times 2$ cases (A) and (E) we obtain the remaining three canonical normal forms.
We will denote by $\text{(X)}\oplus\text{(Y)}$ the canonical form that corresponds to the direct sum of the quadratic forms. We then multiply on the left by $-J$ and obtain the corresponding matrix in $\mathfrak{sp}(2,\RR)$. We obtain three further cases, namely: \\
for $\text{(A)} \oplus \text{(A)}$:
\begin{equation*}
  D_1 
= 
  \begin{pmatrix}
    a_1 & 0   & 0    & 0    \\
    0   & a_2 & 0    & 0    \\
    0   & 0   & -a_1 & 0    \\
    0   & 0   & 0    & -a_2
  \end{pmatrix}, 
\quad 
  a_1 \geq a_2 \geq 0,
\end{equation*}
for $\text{(E)} \oplus \text{(A)}$: 
\begin{equation*}
  D_5
=
  \begin{pmatrix}
     0       & 0 & \eps &  0 \\
     0       & a & 0    &  0 \\
    -b^2\eps & 0 & 0    &  0 \\
     0       & 0 & 0    & -a
  \end{pmatrix},
\qquad 
  a\geq 0,\;b\geq0,\; \eps=\pm 1,
\end{equation*}
and for $\text{(E)} \oplus \text{(E)}$:
\begin{equation*}
  D_6
=
  \begin{pmatrix}
     0         &  0         & \eps & 0    \\
     0         &  0         & 0    & \eta \\
    -b_1^2\eps &  0         & 0    & 0    \\
     0         & -b_2^2\eta & 0    & 0
  \end{pmatrix},
\qquad 
  b_1\geq b_2\geq0,\;(\eps,\eta)\in\{(1,1),(1,-1),(-1,-1)\}. 
\end{equation*}

We summarize our specifications of the results in \cite{Arn89,Wil36} for $d=2$:
\begin{corollary} \label{wil}
  For any $X\in\sptwor$, 
  there exists $B \in\Sp(2,\RR)$ 
  such that $\ad( B) X \in \mathcal{N}$,
  where $\mathcal{N} :=  \{D_k: k = 1,\ldots,7\}$.
\end{corollary}
%
\subsection{Embedding result} \label{subsec:main_result}
%
To prove Theorem~\ref{kommt_noch},  we first observe that $\SS^+$ is
simply connected so that we can pass to its Lie algebra denoted by $\gh_e^+$, see
also Lemma~\ref{lem:HS}.
Let 
$g_*^+$, $\kappa_*^+$ and  $(R_{-1})_{*}$ 
be the tangent maps corresponding to 
$g^+$, $\kappa^+$ and  $R_{-1}$ in Theorem~\ref{kommt_noch}. 
Note that $g_*^+$, $\kappa_*^+$ 
are Lie algebra embeddings of $\gh_e^+$  into the  $\sptwor$, whereas 
 $(R_{-1})_{*} $ is a Lie algebra isomorphism of $\gh_e^+$.  
We have to prove that there exists $B\in\Sp(2,\RR)$ and a Lie algebra isomorphism
$\Phi\colon\gh_e^+\to\gh_e^+$ satisfying  
\begin{equation} \label{nochmal} 
(R_{-1})_{*}  \Phi = \Phi (R_{-1})_{*} 
\end{equation}
such that 
\begin{equation} \label{to_show_1} 
 g^+_*= B (\kappa_*^+\Phi) B^{-1}.
\end{equation}

The image of $\gh_e^+$ under $\kappa_*^+$ is the Lie algebra of $\TDS(2)$ and,
with slight abuse of notation, we identify it with $\gh_e^+$. 
By taking the derivative of the following four one-parameters subgroups,
\begin{align*}
\eta\mapsto \kappa^+(\exp(2\eta),0,0,0) ,\;
\eta\mapsto \kappa^+(1,\eta,0,0), \;
\eta\mapsto \kappa^+(1,0,2\eta,0) , \;
\eta\mapsto \kappa^+(1,0,0,2\eta),
\end{align*}
we get a  basis of $\gh_e^+$, namely
\begin{equation}  \label{initial} 
D^+=- H_{1,0} + (1-2\gamma) H_{0,1}\qquad P^+=X_{-\alpha},\qquad
Q^+=-X_{-\alpha-\beta}, \qquad T^+=-X_{-2\alpha-\beta}.
\end{equation}
By Table~\ref{tab:commutator_relations}, the  non-zero brackets  of these generators are 
\begin{equation} \label{ext_H_rel_gamma+}
  [D^+,P^+]
=
  2(1-\gamma)
  P^+,
\quad
  [D^+,Q^+]
= 
  2\gamma
  Q^+,
\quad
 [P^+,Q^+]
=
  T^+,
\quad
\quad 
  [D^+,T^+]
=
  2T^+.
\end{equation}
The action of the Lie algebra isomorphism $(R_{-1})_{*}$ is given by
\begin{equation*}
(R_{-1})_{*}  D^+ = D^+, \quad (R_{-1})_{*}  P^+ = P^+,\quad (R_{-1})_{*}  Q^+ = -Q^+, \quad (R_{-1})_{*}  T^+ = -T^+.
\end{equation*}
Observe that, for any fixed $u,z\in\RR$ with $uz\neq 0$,  the
linear map $\Phi\colon\gh_e^+\to\gh_e^+$ defined  by
\begin{equation}
\Phi D^+ := D^+, \quad \Phi P^+ := u P^+, \quad \Phi Q^+ :=  z
Q^+, \quad \Phi T^+ :=  uz T^+\label{eq:3}
\end{equation}
is a Lie algebra isomorphism satisfying  \eqref{nochmal}. 

An arbitrary Lie algebra embedding $g_*^+\colon  \gh_e^+\rightarrow \sptwor$ is in one-to-one
correspondence with four linearly independent generators $D,P,Q,T$ of $\sptwor$
whose Lie brackets are given by
\begin{equation} \label{ext_H_rel_gamma_g}
  [D,P]
=
  2(1-\gamma)
  P,
\quad
  [D,Q]
= 
  2\gamma
  Q,\
\quad 
  [P,Q]
=
  T,
\quad 
  [D,T]
=
  2T,
\quad 
  [P,T] 
= 
  [Q,T] 
= 
  0
\end{equation}
where $D,P,Q,T$ are the images of $D^+,P^+,Q^+,T^+$ and hence determine $g_*^+$.
Then, by \eqref{to_show_1}, it remains to prove the following theorem.
%
\begin{theorem} \label{thm:main}
For $\gamma \in (0,1)\setminus\{\frac{1}{3}, \frac{2}{3}\}$, let $D^+,P^+,Q^+,T^+$ be given by \eqref{initial}.
Then, for any fixed generators $D,P,Q,T \in \sptwor$ fulfilling \eqref{ext_H_rel_gamma_g}, 
there exists $B\in\Sp(2,\RR)$
and  a Lie algebra isomorphism
$\Phi\colon\gh_e^+\to\gh_e^+$ satisfying  \eqref{nochmal} so that
\begin{equation*}
\ad(B)\Phi D^+= D, \quad \ad(B)\Phi P^+= P, \quad \ad(B)\Phi Q^+=
Q, \quad \ad(B)\Phi T^+= T.\label{eq:1}
\end{equation*}
\end{theorem}
%
Note that, the choice of a map $\Phi$ as in~\eqref{eq:3} allows to
change $P$ and $Q$ up to a multiplicative constant.
\begin{proof}
First we obtain by straightforward computation 
\begin{enumerate}[(i)]
 \item 
  for $ 0< \gamma \leq \tfrac{1}{2}$, 
  $\Phi$  in \eqref{eq:3} with $u=z=-1$ and 
  $
    B 
  :=
    \left(
    \begin{smallmatrix}
      0 & 0 & -1 & 0 \\
      0 & 1 &  0 & 0 \\
      1 & 0 &  0 & 0 \\
      0 & 0 &  0 & 1
    \end{smallmatrix}
    \right)
  \in
    \Sp(2,\RR)
  $
that
\begin{equation}\label{eq:representation_extended_Heisenbergalgebra_gamma0}
\begin{array}{ll}
    \ad(B)\Phi D^+ 
  = 
    H_{1,0} + (1-2\gamma) H_{0,1}, 
  &\; 
    \ad(B)\Phi P^+
  = 
    X_{\alpha+\beta}, 
  \\
    \ad(B)\Phi Q^+
  =
    X_{\alpha}, 
  &\;
    \ad(B)\Phi T^+
  = 
    -X_{2\alpha+\beta}.
\end{array}
\end{equation}
\item 
  for $\tfrac{1}{2} \leq\gamma < 1$, 
  $\Phi$ in \eqref{eq:3} with $u=1$, $z=-1$ and
  $B := - J$
  that
  \begin{equation}\label{eq:representation_extended_Heisenbergalgebra_gamma1}
  \begin{array}{ll}
      \ad(B)\Phi D^+
    =
      H_{1,0} - (1-2\gamma) H_{0,1},
    &\;
      \ad(B)\Phi P^+ 
    =
      X_{\alpha},
    \\
      \ad(B)\Phi Q^+  
    =
      X_{\alpha + \beta},
    &\;
      \ad(B)\Phi T^+ 
    =
      X_{2\alpha+\beta}.
  \end{array}
  \end{equation}
\end{enumerate}                 
We show that, up to conjugation and change of sign in the
definition of $P$ and $Q$, the matrices in 
\eqref{eq:representation_extended_Heisenbergalgebra_gamma0} 
and 
\eqref{eq:representation_extended_Heisenbergalgebra_gamma1}
are the unique ones which fulfill \eqref{ext_H_rel_gamma_g}.

Let $D,P,Q,T\in \sptwor$ be arbitrarily fixed, linearly independent matrices with property \eqref{ext_H_rel_gamma_g}.
By Corollary~\ref{wil}, up to conjugation, $D$ is one of the canonical
forms in $\mathcal{N}$.   
For each $D \in \mathcal{N}$ we have to find $P, Q\in\sptwor$
fulfilling in particular $[D,P] = 2(1-\gamma) P$ and $[D,Q] = 2\gamma Q$.
For this purpose we consider for all $\gamma \in (0,1)$ the pairs of vector spaces
\begin{equation*}
  V_\Gamma
:=
  \{
    X\in\sptwor
    :
    [D,X]=\Gamma X 
  \}, 
\qquad 
  \Gamma \in \{2(1-\gamma), 2\gamma\},
  \end{equation*}
which coincide in the case $\gamma = \tfrac{1}{2}$.
We compute the Lie brackets of every $D \in \mathcal{N}$ 
with the basis elements in \eqref{basis},
i.e., we use Table~\ref{tab:commutator_relations} to find 
$[D,B_k] = \sum_{j=1}^{10} d_{k j} B_j$, $k=1,\ldots,10$.
Then we obtain for any 
$X := \sum_{k=1}^{10} x_k B_k \in \sptwor$  that
\begin{align*}
  [D,X] - \Gamma X 
&= 
  \sum_{k=1}^{10} x_k [D,B_k] - \Gamma \sum_{j=1}^{10} x_j B_j
\\ 
&=
  \sum_{k=1}^{10} 
  \sum_{j=1}^{10} 
    x_k d_{k j} B_j - \Gamma \sum_{j=1}^{10} x_j B_j
\\
&=
  \sum_{j=1}^{10} \left( \sum_{k=1}^{10} (d_{k j} - \Gamma\delta_{kj}) x_k \right) B_j
\end{align*}
where $\delta_{kj} = 1$ for $k=j$ and $\delta_{kj}=0$ otherwise
(Kronecker delta).
Hence $[D,X] = \Gamma X$ is equivalent to $M_\Gamma x = 0$, where $x := (x_k)_{k=1}^{10}$ and
\begin{equation} \label{matrix_M}
  M_\Gamma 
:= 
  (d_{k j} - \Gamma \delta_{kj})_{j,k=1}^{10}.
\end{equation}
To have non-trivial linearly independent solutions $X$ (for $P$ and $Q$) we need that 
$M_\Gamma$ has rank $\leq 9$ for each $\Gamma \in  \{2(1-\gamma), 2\gamma\}$ if $\gamma \not = \frac{1}{2}$
and rank $\leq 8$ if $\gamma = \frac{1}{2}$.
For the seven matrices $D \in \mathcal{N}$ the 
brackets $[D,B_k]$, $k=1,\ldots,10$, 
the corresponding matrices $M_\Gamma$ and their determinants 
are listed in the appendix. 
Using these computations we have the following cases:
\begin{enumerate}[1.]
\item
  For $D \in \{D_4,D_6,D_7\}$ we see immediately that $\det \, M_\Gamma \neq 0$ for $\Gamma \in  \{2(1-\gamma), 2\gamma\}$, $\gamma \in (0,1)$ 
  so that the matrices have full rank.
\item
  For $D \in \{D_2,D_3\}$ only $2a = 2(1-\gamma) = 2 \gamma$ leads to $\det \, M_\Gamma = 0$, $\Gamma \in \{2(1-\gamma), 2\gamma\}$.
  This implies $\gamma = \tfrac{1}{2}$ and $M_{2\gamma}=M_{2(1-\gamma)}= M_1$. But $M_1$ has rank $9$ in both cases $D \in \{D_2,D_3\}$.
\item
  For $D=D_5$ we obtain $\det \, M_\Gamma = 0$ for $\Gamma \in \{2(1-\gamma), 2\gamma\}$ in the following cases:
  \begin{enumerate}[{3}.1]
  \item 
    $2a = 2(1-\gamma) = 2\gamma$ or $a = 2(1-\gamma) = 2\gamma$ and $b=0$, which implies $\gamma =  \tfrac{1}{2}$. 
    However, as in the previous case, $M_1$ has rank $9$.
  \item
    For $b = 0$: $2a = 2 (1-\gamma)$ and $a = 2\gamma$ or $2a = 2 \gamma$ and $a = 2(1-\gamma)$
    which is only possible if $\gamma = \tfrac{1}{3}$ or $\gamma = \tfrac{2}{3}$. 
    But for these cases the second or third column of $M_\Gamma$ is zero 
    so that $P$ would be a multiple of $X_\beta$ and $Q$ a multiple of $X_{\alpha+\beta}$ (or vice versa) but these basis elements commute.
  \end{enumerate}
\item
  Finally, for $D = D_1$, the matrix $M_\Gamma$ is a diagonal matrix with diagonal entries
  \begin{small}
  \begin{equation*}
  \left(
    a_1 - a_2 - \Gamma,
    2a_2 - \Gamma,
    a_1 + a_2 - \Gamma,
    2a_1 - \Gamma,
    a_2 - a_1 - \Gamma,
    - 2a_2 - \Gamma,
    - a_2 - a_1 - \Gamma,
    - 2a_1 - \Gamma,
    - \Gamma,
    - \Gamma
  \right).
  \end{equation*}
  \end{small}
  \begin{table}[htbp]
    \centering
    \begin{tabular}{l|r|rrr} \toprule
      $[\,\cdot\,,\,\cdot\,]$ & $X_\alpha$             & $X_\beta$          & $X_{\alpha+\beta}$    & $X_{2\alpha + \beta}$ \\\midrule
      $X_\alpha$              & $0$                    & $X_{\alpha+\beta}$ & $X_{2\alpha + \beta}$ & $0$                   \\\midrule
      $X_\beta$               & $-X_{\alpha+\beta}$    & $0$                & $0$                   & $0$                   \\
      $X_{\alpha+\beta}$      & $-X_{2\alpha + \beta}$ & $0$                & $0$                   & $0$                   \\
      $X_{2\alpha + \beta}$   & $0$                    & $0$                & $0$                   & $0$                   \\\bottomrule
    \end{tabular}
    \caption{Commutator relations $[B_i,B_j]$ for $B_i,B_j \in\cB$, $i,j=1,\ldots,4$.}
    \label{tab:commutator_relations_positive}
  \end{table}
  Since $a_1\geq a_2\geq 0$ and $\Gamma > 0$, the last six elements are less than zero so that
  $M_\Gamma x = 0$ implies $x_5 = x_6 = \ldots = x_{10} = 0$. 
  Hence any solution $X$ is a linear combination of at most the first four basis elements.
  To obtain solutions $X$ for $P$ and $Q$ such that $[P,Q] = T \not = 0$ 
  we see from Table~\ref{tab:commutator_relations_positive} that at least one solution $X$ must be
  a nontrivial combination with $B_1 = X_\alpha$, i.e., $x_1\neq 0$.
  Consequently, we need $a_1 - a_2 - \Gamma = 0$ for one $\Gamma \in \{2(1-\gamma), 2\gamma\}$.

  Let $\gamma\neq\frac{1}{2}$.
  Then, for the other choice of $\Gamma$, another diagonal element has to be zero.
  Table~\ref{tab:solutions} shows the corresponding six cases. Note that by setting the first and another diagonal element to zero,
  the values $a_1$ and $a_2$ are uniquely determined by $\gamma$.
  \begin{table}[htbp]
    \centering
  {\small
    \begin{tabular}{ccc|cccc} \toprule
      $a_1$ & $a_2$ & $\Gamma$ & $a_1-a_2-\Gamma$ & $2a_2-\Gamma$ & $a_1+a_2-\Gamma$ & $2a_1-\Gamma$ \\\midrule
      \multirow{2}{*}{$\gamma + 1$} & \multirow{2}{*}{$1 - \gamma$} & $2(1-\gamma)$ & $2 (2 \gamma - 1)$ & $0$ & $2 \gamma$ & $4 \gamma$ \\\cmidrule{3-7}
      & & $2\gamma$ & $0$ & $-2 (2 \gamma - 1)$ & $-2 (\gamma - 1)$ & $2$ \\\midrule
      \multirow{2}{*}{$1$} & \multirow{2}{*}{$1 - 2 \gamma$} & $2(1-\gamma)$ & $2 (2 \gamma - 1)$ & $-2 \gamma$ & $0$ & $2 \gamma$ \\\cmidrule{3-7}
      \multicolumn{2}{c}{\footnotesize($\gamma < \frac{1}{2}$)} & $2\gamma$ & $0$ & $-2 (3 \gamma - 1)$ & $-2 (2 \gamma - 1)$ & $-2 (\gamma - 1)$ \\\midrule
      \multirow{2}{*}{$1 - \gamma$} & \multirow{2}{*}{$1 - 3 \gamma$} & $2(1-\gamma)$ & $2 (2 \gamma - 1)$ & $-4\gamma$ & $-2 \gamma$ & $0$ \\\cmidrule{3-7}
      \multicolumn{2}{c}{\footnotesize($\gamma < \frac{1}{3}$)} & $2\gamma$ & $0$ & $-2 (4 \gamma - 1)$ & $-2 (3 \gamma - 1)$ & $-2 (2 \gamma - 1)$ \\\midrule
      \multirow{2}{*}{$2 - \gamma$} & \multirow{2}{*}{$\gamma$} & $2(1-\gamma)$ & $0$ & $2 (2 \gamma - 1)$ & $2 \gamma$ & $2$ \\\cmidrule{3-7}
      & & $2\gamma$ & $-2 (2 \gamma - 1)$ & $0$ & $-2 (\gamma - 1)$ & $-4(\gamma - 1)$ \\\midrule
      \multirow{2}{*}{$1$} & \multirow{2}{*}{$2 \gamma - 1$} & $2(1-\gamma)$ & $0$ & $2 (3 \gamma - 2)$ & $2 (2 \gamma - 1)$ & $2 \gamma$ \\\cmidrule{3-7}
      \multicolumn{2}{c}{\footnotesize($\gamma > \frac{1}{2}$)} & $2\gamma$ & $-2 (2 \gamma - 1)$ & $2 (\gamma - 1)$ & $0$ & $-2 (\gamma - 1)$ \\\midrule
      \multirow{2}{*}{$\gamma$} & \multirow{2}{*}{$3 \gamma - 2$} & $2(1-\gamma)$ & $0$ & $2 (4 \gamma - 3)$ & $2 (3 \gamma - 2)$ & $2 (2 \gamma - 1)$ \\\cmidrule{3-7}
      \multicolumn{2}{c}{\footnotesize($\gamma > \frac{2}{3}$)} & $2\gamma$ & $-2 (2 \gamma - 1)$ & $4 (\gamma - 1)$ & $2 (\gamma - 1)$ & $0$
      \\\bottomrule
    \end{tabular}
  }
    \caption{Possible solutions for $a_1$ and $a_2$ such that $a_1-a_2-\Gamma = 0$ (the first entry of $M_\Gamma$).}
    \label{tab:solutions}
  \end{table}
  %

  The pairs $(a_1,a_2) \in \{(1-\gamma,1-3\gamma),(\gamma,3\gamma-2)\}$ lead to a solution $X$
  which is a multiple of $X_{2\alpha+\beta}$ and consequently commutes with all four basis elements, see Table~\ref{tab:commutator_relations_positive}.
  This contradicts the requirement $[P,Q] = T \not = 0$.

  For $(a_1,a_2) \in \{(\gamma + 1,1 -\gamma), (2-\gamma,\gamma)\}$ the solutions $X$ (for $P$ and $Q$) are multiples of $X_\alpha$ and $X_\beta$,
  so that by Table~\ref{tab:commutator_relations_positive}, the matrix $T$ becomes a multiple of $X_{\alpha + \beta}$. But this $T$ cannot commute with $P$ and $Q$
  as required by  \eqref{ext_H_rel_gamma+}.

  For $(a_1,a_2) = (1,1-2\gamma)$ with  $\gamma < \tfrac{1}{2}$, $\gamma \not=\frac{1}{3}$
  we obtain (up to multiplication with scalars) the setting \eqref{eq:representation_extended_Heisenbergalgebra_gamma0},
  and for $(a_1,a_2) = (1,2\gamma-1)$ with $\gamma > \tfrac{1}{2}$, $\gamma \not=\frac{2}{3}$ the result \eqref{eq:representation_extended_Heisenbergalgebra_gamma1}.
  For $\gamma \in \left\{\tfrac{1}{3},\tfrac{2}{3}\right\}$ see Remark~\ref{ein_drittel}.

  Let $\gamma=\frac{1}{2}$ which implies $\Gamma = 1$.
  Then, regarding that $x_1 \not = 0$, and consequently $a_1 - a_2 - 1 = 0$, i.e., $a_2 = a_1 - 1$, 
  the first four diagonal elements of $M_1$ must read as
  \begin{equation*}
    (0, 2a_1 - 3, 2a_1 - 2, 2a_1 -1). 
  \end{equation*}  
  The matrix $M_1$ must have rank $\le 8$ and $a_2 \ge 0$. This is only possible if
  $(a_1,a_2) \in \{ (\tfrac32,\tfrac{1}{2}), (1,0) \}$.
  For $(a_1,a_2) = (\tfrac32,\tfrac{1}{2})$ the solutions $X$ for $P$ and $Q$ are multiples of $X_\alpha$ and $X_\beta$.
  But then $T=[P,Q]$ is a multiple of $X_{\alpha + \beta}$ which does not commute with both $X_\alpha$ and $X_\beta$
  as required by \eqref{ext_H_rel_gamma+}.

  For $(a_1,a_2) = (1,0)$ we obtain the solution
  \begin{equation*}
    D
  =
    H_{1,0} 
  \qquad 
    P
  =u  X_{\alpha} + v  X_{\alpha+\beta},
  \qquad 
    Q
  =  w X_{\alpha} + z X_{\alpha+\beta},
  \qquad 
    T
  = (uz-vw)
    X_{2\alpha+\beta},
  \end{equation*}
  where  $uz-wv\neq 0$ must be fulfilled.  Possibly changing $P$ into
  $-P$, we can assume that $uz-wv>0$. 
  Now straightforward computation shows that
  \begin{equation*}
    C
  := 
    \frac{1}{m}
    \begin{pmatrix}
      m^2 &  0 & 0&  0 \\
      0   &  z & 0& -v \\
      0   &  0 & 1&  0 \\
      0   & -w & 0&  u
    \end{pmatrix}, 
  \qquad 
    m 
  := 
    \sqrt{uz-vw}
  \end{equation*}
  is a symplectic matrix which fulfills
  \begin{align*}
    \ad(C) H_{1,0}= H_{1,0} ,\quad
    \ad(C) X_\alpha =  u X_\alpha + v X_{\alpha+\beta}, \quad
    \ad(C)X_{\alpha+\beta}  =  w X_\alpha + z X_{\alpha+\beta}.
  \end{align*}
\end{enumerate}
This finishes the proof.
\end{proof}
%
\begin{remark}\label{ein_drittel}
For $\gamma=\frac{1}{3}$ (and $\gamma=\frac{2}{3}$) there are 'non-standard'
embeddings $\widetilde{\kappa}$ of 
$\SS^+$ into $\Sp(2,\RR)$, which are not conjugated
with~\eqref{eq:definition_g_+} unless $w=0$.  At the Lie algebra level, 
$\widetilde{\kappa}_*\colon\gh_e^+\to\sptwor$
acts as
\begin{equation*}
  \widetilde{\kappa}_*D^+= H_{1,0} + \frac{1}{3}H_{0,1}
\qquad
  \widetilde{\kappa}_*P^+=u X_{\alpha + \beta} 
\qquad
  \widetilde{\kappa}_*Q^+=v X_{\alpha} + wX_{\beta}
\qquad 
  \widetilde{\kappa}_*Z^+= -uvX_{2\alpha+\beta}.
\end{equation*}
However, it is easy to check that for such embeddings there does not
exist a symplectic matrix $A$ satisfying~\eqref{eq:Aconjugation}.
For further explanations we refer to \cite{Hae14}.
\end{remark}

%
\section{Coorbit spaces for equivalent representations} \label{sec5}
In this section we show the relation between the coorbit spaces of isomorphic groups with equivalent representations
and apply it to our setting.
We briefly introduce the general coorbit space theory and describe how isomorphic groups with equivalent representations lead to isomorphic
scales of coorbit spaces.
Then we specify the results for our connected shearlet and shearlet Toeplitz groups and their isomorphic subgroups of the symplectic group.
Since the latter ones are equipped with a metaplectic representation this leads finally to metaplectic coorbit spaces.
%
\subsection{Coorbit spaces} \label{subsec:coorbit_spaces}
Let $G$ be a locally compact group with left Haar measure $d \mu$.
A \emph{unitary representation} of $G$ on a Hilbert space $\cH$ is a homomorphism $\pi$ from $G$ 
into the group $\cU(\cH)$ of unitary operators on $\cH$ 
that is continuous with respect to the strong operator topology, see \cite{Fol95}. 
A representation is called \emph{irreducible} if there does not exist a nontrivial $\pi$-invariant subspace of $\cH$.
A unitary, irreducible representation $\pi$ fulfilling
\begin{equation}\label{eq:admissibility}
  \int_G \lvert\langle \psi,\pi(g)\psi\rangle\rvert^2
  d\mu (g)
<
  \infty
\end{equation}
for some $\psi\in\cH$ is called \emph{square integrable} and a function fulfilling \eqref{eq:admissibility} \emph{admissible}. 
Assume that there exists a square integrable representation $\pi$ of $G$.
For an admissible function $\psi\in\cH$ the mapping $V_\psi\colon\cH\to L_2(G)$ with 
$V_\psi(f)(g) := \langle f,\pi(g)\psi\rangle$ 
is known as  \emph{voice transform} of $f$ (with respect to $\psi$).
The admissibility condition \eqref{eq:admissibility} is important since it yields 
to a resolution of the identity that allows the reconstruction of a function $f\in\cH$ from its voice transform 
$(\langle f,\pi(g)\psi \rangle)_{g\in G}$. 
Using the voice transform we can reformulate the admissibility condition \eqref{eq:admissibility} as $V_\psi(\psi)\in L_2(G)$.

For a general real-valued weight $\mathrm{w}$ and $1 \leq p \leq \infty$ we define the weighted $L_p$ space on $G$ as
\begin{equation*}
  L_{p,\mathrm{w}}(G)
:=
  \{
    F\text{ measurable} : F \mathrm{w}\in L_p(G)
  \}.
\end{equation*}
Further we will need the weighted sequence spaces 
\begin{equation*}
  \ell_{p,{\rm w}}
:=
  \{
    (c_i)_{i\in\cI}
    :
    (c_i \mathrm{w_i})_{i\in\cI}
    \in
    \ell_p
  \}.
\end{equation*}
The voice transform can be extended from the Hilbert space ${\mathcal H}$  to weighted Banach spaces of distributions using coorbit space
theory.
This theory was developed by Feichtinger and Gr\"ochenig in a series of papers \cite{FG88,FG89a,FG89b,Gro91,GP09} and we collect the basic ideas
in the following.
Let now $w$ be a real-valued, continuous and submultiplicative weight on $G$, i.e.,
$w(gh) \leq w(g)w(h)$ for all $g,h \in G$ fulfilling in addition the 
conditions stated in \cite[Section 2.2]{Gro91}.
We assume that the so-called \emph{set of analyzing vectors} 
\begin{equation} \label{A_w}
  \cA_w
:=
  \{ \psi\in\cH : V_\psi(\psi)\in L_{1,w}(G) \}.
\end{equation}
is nonempty and fix a nontrivial function $\psi\in \cA_w$.
We can define a set of \emph{test functions} and equip it with a norm such that it becomes a Banach space by
\begin{equation*}
  \cH_{1,w}
:=
  \{ f\in\cH : V_\psi(f) \in L_{1,w}(G) \}, \quad \lVert f\rVert_{\cH_{1,w}}:= \lVert V_\psi(f)\rVert_{L_{1,w}(G)}.
\end{equation*}
Its anti-dual space, i.e., the space of all continuous conjugate-linear functionals on $\cH_{1,w}$, also called space of \emph{distributions}, 
is denoted by $\cH_{1,w}^{\sim}$.
The definitions of $\cH_{1,w}$ and $\cH_{1,w}^{\sim}$ are independent
of the choice of the analyzing vector $\psi\in\cA_w$, see \cite[Lemma
4.2]{FG88}, in particular $\cH_{1,w} = \cA_w$ as sets.
The spaces $\cH_{1,w}$ and $\cH_{1,w}^{\sim}$ are $\pi$-invariant Banach spaces with continuous embeddings $\cH_{1,w}\hookrightarrow\cH\hookrightarrow\cH_{1,w}^{\sim}$.
The inner product on $\cH\times \cH$ extends to a sesquilinear form on $\cH_{1,w}^{\sim}\times\cH_{1,w}$: for $\psi\in\cH_{1,w}$ and $f\in\cH_{1,w}^{\sim}$ the extended representation coefficients
\begin{equation*}
  V_\psi(f)(g)
:=
  \langle
    f,\pi(g)\psi
  \rangle_{\cH_{1,w}^{\sim}\times\cH_{1,w}}
\end{equation*}
are well-defined and provide the desired generalization of the \emph{voice transform on} $\cH_{1,w}^{\sim}$.

Let $m$ be a $w$-moderate weight on $G$ which means that $m(xyz) \leq w(x)m(y)w(z)$ for all $x,y,z\in G$.
The \emph{coorbit space} of $L_{p,m}(G)$ is given by
\begin{equation*}
  \Co(L_{p,m}(G))
:=
  \cH_{p,m}
:=
  \{
    f\in\cH_{1,w}^{\sim} 
    :
    V_\psi(f)
    =
    \langle 
    f,\pi(\,\cdot\, )\psi 
    \rangle_{\cH_{1,w}^{\sim} \times \cH_{1,w}} \in L_{p,m}(G)\}
\end{equation*}
with norm 
$\lVert f\rVert_{\cH_{p,m}} 
= 
\lVert \langle 
f,\pi(\,\cdot\, )\psi 
\rangle_{\cH_{1,w}^{\sim} \times \cH_{1,w}}\rVert_{L_{p,m}(G)}$.
It is a \emph{$\pi$-invariant Banach space} which does not depend
on the choice of the analyzing vector $\psi\in \cA_w$,
see \cite[Theorem 4.2]{FG89a}. 
In particular, the spaces $\cH$, $\cH_{1,w}$ and $\cH_{1,w}^{\sim}$ 
can be identified with the following coorbit spaces
\begin{equation*}
  \cH
=
  \Co(L_2(G)),
\quad
  \cH_{1,w}
=
  \Co(L_{1,w}(G)),
\quad\text{and}\quad
  \cH_{1,w}^{\sim}
=
  \Co(L_{\infty,\frac{1}{w}}(G)).
\end{equation*}
To establish atomic decompositions and Banach frames for coorbit spaces we need 
(i) a stronger integrability condition for the analyzing functions than \eqref{A_w}, and (ii)
a reasonable discretization of our group $G$.
Concerning (i) we require that the following \emph{better subset} (or \emph{set of basic atoms}) is nontrivial
\begin{equation*}
  \cB_w
:=
  \{
    \psi\in\cH
    :
    V_\psi(\psi)\in\cW^{L}(L_\infty(G),L_{1,w}(G))
  \}
\end{equation*}
where 
\begin{equation*}
  \cW^{L}(L_\infty(G),L_{1,w}(G)) 
:= 
  \{ 
    F\in L_{\infty,\text{loc}} 
  : 
    H_F
    \in 
    L_{1,w}(G)
  \}
\end{equation*}
and $H_F\colon G\to\RR$ is given by 
$
  H_F(x)
:=
 \lVert (L_x\chi_Q) F\rVert_\infty 
= 
  \sup_{y\in xQ} \lvert F(y)\rvert 
$,
with $Q$ being a relatively compact neighborhood of the identity element $e\in G$. Then we choose $0 \not = \psi \in \cB_w$.
With respect to (ii) we assume that $G$ can be discretized
on a so-called \emph{well-spread set}. A (countable) family $X = \{g_i : i \in \cI \}$ in $G$ is called \emph{well-spread}
if $\bigcup_{i\in\cI} g_i U = G$ for some compact set $U$ with non-void interior, and
 if for all compact sets $K\subset G$ there exists a constant $C_K$ such that
  \begin{equation*}
    \sup_{j\in\cI} \#\{i\in \cI : g_i K\cap g_j K \neq \emptyset\}\leq C_K.
  \end{equation*}
The following theorem collects results about the existence of atomic decompositions and Banach frames from \cite{Gro91,DKST09,DHST13}.
\begin{theorem}
  Let $1 \leq p \leq \infty$ and
  $\psi \in {\mathcal B}_w$, $\psi \neq 0$.
  Then there exists a (sufficiently small) neighborhood $U\subset G$ of $e$ such that for
  any well-spread set $X=\{g_i : i\in \cI\}$ in $G$
  the set
  $\{\pi(g_i) \psi : i\in\cI\}$ provides an \emph{atomic decomposition} and a \emph{Banach frame} for
  $\cH_{p,m}$.
  \\
  \emph{Atomic decomposition.}
    Every
    $f \in \cH_{p,m}$ possesses an expansion
    \begin{equation*}
      f 
    = 
      \sum_{i \in \cI} c_i(f)\pi(g_i) \psi,
    \end{equation*}
    where the sequence of coefficients $(c_i(f))_{i\in\cI}$ depends linearly on $f$ and satisfies
    \begin{equation*}
      \lVert (c_i (f))_{i \in \cI}\rVert_{\ell_{p,m}}
    \leq C
      \lVert f\rVert_{\cH_{p,m}}
    \end{equation*}
    with a constant $C$ only depending on $\psi$.
    Conversely, if $(c_i)_{i \in \cI} \in \ell_{p,m}$,
    then $f=\sum_{i \in \cI} c_i \pi(g_i) \psi$ is in
    $\cH_{p,m}$  and
    \begin{equation*}
      \lVert f\rVert_{\cH_{p,m}} 
    \leq 
      C'\lVert (c_i)_{i \in \cI}\rVert_{\ell_{p,m}}.
    \end{equation*}
	\emph{Banach frames.}
    The set $\{\pi(g_i)\psi: i \in \cI \}$
    is a  Banach frame for $\cH_{p,m}$ which means that
     there exist two constants $C_1, C_2>0$ depending only on $\psi$ such that
    \begin{equation*}
      C_1\lVert f\rVert_{\cH_{p,m}} 
    \leq
      \lVert
        (
          \langle 
            f, \pi(g_i)\psi
          \rangle_{\cH_{1,w}^{\sim}\times \cH_{1,w}}
        )_{i \in \cI}
      \rVert_{\ell_{p,m}}
    \leq
      C_2\lVert f\rVert_{\cH_{p,m}},
    \end{equation*}
    and there exists a bounded, linear reconstruction operator
    $\cR$ from $\ell_{p,m}$ to $\cH_{p,m}$ such that
        \begin{equation*}
          \cR\bigl(
            (\langle 
              f,\pi(g_i)\psi 
            \rangle_{\cH_{1,w}^{\sim}\times\cH_{1,w}}
            )_{i\in\cI}
            \bigr)
        =
          f.
        \end{equation*}
  \end{theorem}
%
\subsection{Coorbit spaces for isomorphic groups and equivalent representations} \label{subsec:iso}
%
Let $G$ and $\widetilde{G}$ be locally compact groups with left Haar measures $d\mu$ and $d \widetilde{\mu}$, respectively, 
which are isomorphic with isomorphism $\iota\colon G\to \widetilde{G}$. 
Further, let $\cH$ and $\widetilde{\cH}$ be Hilbert spaces with isometric isomorphism $\Psi\colon \cH \to \widetilde{\cH}$.
Let $\pi\colon G\to\cU(\cH)$ and $\widetilde{\pi}\colon \widetilde{G}\to\cU(\widetilde{\cH})$ 
be unitary representations of $G$ on $\cH$ and of $\widetilde{G}$ on $\widetilde{\cH}$, respectively,
so that for all $g \in G$ and all $f \in \cH$
\begin{equation} \label{main_iso}
  \Psi(\pi(g) f) 
= 
  \widetilde{\pi} \left( \iota(g) \right) (\Psi f) .
\end{equation}
We will refer to such $\pi$ and $\widetilde{\pi}$ as equivalent representations.
Setting $\widetilde{f} := \Psi f$, i.e., $f = \Psi^{-1} \widetilde{f}$ this can be rewritten as
$
\Psi \left( \pi(g) \Psi^{-1} \widetilde{f} \right) = \widetilde{\pi} \left( \iota(g) \right) \widetilde{f}
$.
The following diagram illustrates the relations.
\begin{center}
\begin{tikzpicture}[node distance = 3cm, auto]
  \node (UH) {$\mathcal{U}(\mathcal{H})$};
  \node (UTH) [right = 3.5cm of UH] {$\mathcal{U}(\widetilde{\mathcal{H}})$};
  \node (G) [below =2.1cm of UH] {$G$};
  \node (TG) [below =2cm of UTH] {$\widetilde{G}$};
  \draw[->] (G) to node {$\pi$} (UH);
  \draw[->] (UH) to node {$\Psi(\,\cdot\,)\Psi^{-1}$} (UTH);
  \draw[->] (G) to node {$\iota$} (TG);
  \draw[->] (TG) to node [swap] {$\widetilde{\pi}$} (UTH);
\end{tikzpicture}
\end{center}
Assume that $G$, $\cH$ and $\pi$ give rise to 
a sequence of coorbit spaces including atomic decompositions 
and Banach frames as described in the previous Subsection~\ref{subsec:coorbit_spaces}. 
For the corresponding weights and function spaces we use the notation from Subsection~\ref{subsec:coorbit_spaces}.
We are interested in the relation to the coorbit spaces based on  
$\widetilde{G}$, $\widetilde{\mathcal{H}}$ and $\widetilde{\pi}$ in terms of $\iota$ and $\Psi$.
 
First, we define $\widetilde{w} := w \circ \iota^{-1}$  which is clearly a real-valued, continuous, submultiplicative weight on $\widetilde{G}$ 
satisfying the conditions in \cite[Section 2.2]{Gro91}, since $w$ does.
Let
\begin{equation*}
   \widetilde{\cA}_{\widetilde{w}}
:=
  \left\{
    \widetilde{\psi} \in \widetilde{\cH}
    : 
     \langle 
      \widetilde{\psi},\widetilde{\pi}(\widetilde{g})\widetilde{\psi} 
    \rangle_{\widetilde{\cH}}  
    \in L_{1,\widetilde{w}}(\widetilde{G})
  \right\}.
\end{equation*}
Setting $\widetilde{\psi} := \Psi \psi$ for $\psi \in \cA_w$
we obtain by \eqref{main_iso} and since $\Psi$ is an isometry 
\begin{equation*}
 \langle \psi, \pi(\iota^{-1}(\widetilde{g}))\psi \rangle_\cH
=
\langle \Psi \psi, \Psi \left( \pi(\iota^{-1}(\widetilde{g}))\psi \right)\rangle_{\widetilde{\cH}}
=
  \langle 
    \widetilde{\psi}, \widetilde{\pi}(\widetilde{g})\widetilde{\psi} 
  \rangle_{\widetilde{\cH}} .
\end{equation*}
Thus, $\cA_w$ and $\widetilde{\cA}_{\widetilde{w}}$ are isomorphic. 
For an analyzing vector 
$\widetilde{\psi} = \Psi \psi \in \widetilde{\cA}_{\widetilde{w}}$
we introduce the set of \emph{test functions}
\begin{equation*}
  \widetilde{\cH}_{1,\widetilde{w}}
:=
  \left\{ 
    \widetilde{f}\in\cH 
    : 
    \langle \widetilde{f}, \widetilde{\pi}(\widetilde{g})\widetilde{\psi} \rangle 
    \in L_{1,\widetilde{w}}(\widetilde{G})
  \right\}.
\end{equation*}
Since 
$
  \lVert f \rVert_{\cH_{1,w}} 
=
  \lVert \langle f,\pi(g)\psi \rangle \rVert_{L_{1,w}(G)}
=
  \lVert \langle \Psi f, \widetilde{\pi} (\iota(g) ) \Psi \psi \rangle \rVert_{L_{1,w}(G)}
$
we see with $\widetilde{g} = \iota(g)$ and $\widetilde{\psi} = \Psi \psi$ that 
\begin{equation*}
  \lVert f \rVert_{\cH_{1,w}} 
= 
  \lVert \langle f,\pi(g)\psi \rangle \rVert_{L_{1,w}(G)}
=
  \lVert \langle \widetilde{f},\widetilde{\pi}(\widetilde{g})\widetilde{\psi} \rangle \rVert_{L_{1,\widetilde{w}}(\widetilde{G})}
=
  \lVert \widetilde{f} \rVert_{\widetilde{\cH}_{1,\widetilde{w}}} .
\end{equation*}
As for the set of analyzing vectors $\Psi$ induces an isomorphism from $\cH_{1,w}$ onto $\widetilde{\cH}_{1,\widetilde{w}}$, where 
$\widetilde{f} = \Psi f$. 
Let $(\widetilde{\cH}_{1,\widetilde{w}})^{\sim}$ denote the anti-dual space of  $\widetilde{\cH}_{1,\widetilde{w}}$.
This space is related to  $\cH_{1,w}^{\sim}$ by $\Psi^* \widetilde{f} = f$ for all 
$\widetilde{f} \in (\widetilde{\cH}_{1,\widetilde{w}})^{\sim}$,
where 
$\Psi^*\colon (\widetilde{\cH}_{1,\widetilde{w}})^{\sim} \rightarrow \cH_{1,w}^{\sim}$ denotes the adjoint of the isomorphism
$\Psi\colon \cH_{1,w} \rightarrow \widetilde{\cH}_{1,\widetilde{w}}$.
The relations are illustrated in the following diagram.
\begin{center}
\begin{tikzpicture}[node distance = 2cm, auto]
  \node (H1w) {$\cH_{1,w}$};
  \node (H) [right = 2cm of H1w] {$\cH$};
  \node (DH1w) [right = 2cm of H] {$\cH_{1,w}^{\sim}$};
  \node (TH1w) [below =2cm of H1w]{$\widetilde{\cH}_{1,\widetilde{w}}$};
  \node (TH) [right = 2cm of TH1w] {$\widetilde{\cH}$};
  \node (TDH1w) [below = 2cm of DH1w] {$(\widetilde{\cH}_{1,\widetilde{w}})^{\sim}$};
  \draw[->] (H1w) to node {$\Psi$} (TH1w);
  \draw[->] (H) to node {$\Psi$} (TH);
  \draw[->] (TDH1w) to node {$\Psi^*$} (DH1w);
  \draw[right hook->] (H1w) -- (H);
  \draw[right hook->] (H) -- (DH1w);
  \draw[right hook->] (TH1w) -- (TH);
  \draw[right hook->] (TH) -- (TDH1w);
\end{tikzpicture}
\end{center}
The inner product on $\widetilde{\cH}$ extends to a dual pairing on $(\widetilde{\cH}_{1,\widetilde{w}})^{\sim} \times \widetilde{\cH}_{1,\widetilde{w}}$ 
by
\begin{align} 
\langle 
    \widetilde{f}, 
    \widetilde{\pi}(\widetilde{g}) \widetilde{\psi} 
  \rangle_{(\widetilde{\cH}_{1,w})^{\sim} \times \widetilde{\cH}_{1,w}}
  &=
  \langle 
    \widetilde{f}, 
    \Psi \left( \pi(g) \psi \right)
  \rangle_{(\widetilde{\cH}_{1,w})^{\sim} \times \widetilde{\cH}_{1,w}}
  =
  \langle 
    \Psi^* \widetilde{f},
    \pi(g)\psi 
  \rangle_{\cH_{1,w}^{\sim} \times \cH_{1,w}} \nonumber \\
  &=
 \langle 
    f,
    \pi(g)\psi 
  \rangle_{\cH_{1,w}^{\sim} \times \cH_{1,w}}. \label{2_voice}
\end{align}
It follows that both lines of the diagram are isomorphic Gelfand triples 
and the extended voice transforms coincide in the sense of \eqref{2_voice}.
To the $w$-moderate weight function $m$ on $G$ we associate the  $\widetilde{w}$-moderate function $\widetilde{m}$
on $\widetilde{G}$
by $\widetilde{m}(\widetilde{g}) := m(\iota^{-1}(\widetilde{g}))$.
Now we are ready to define the coorbit spaces of $ L_{p,\widetilde{m}}(\widetilde{G})$ by
\begin{align*}
   \widetilde{\cH}_{p,\widetilde{m}}
&:=
  \left\{
    \widetilde{f}\in(\widetilde{\cH}_{1,w})^{\sim} 
    : 
    \langle 
      \widetilde{f},\widetilde{\pi}(\,\cdot\, )\widetilde{\psi} 
    \rangle_
    {(\widetilde{\cH}_{1,w})^{\sim} \times \widetilde{\cH}_{1,w}} \in L_{p,\widetilde{m}}(\widetilde{G})
  \right\}
\end{align*}
which are isomorphic to $\cH_{p,m}$ with isomorphism $\widetilde{f} = \Psi f$. 

Finally, we want to analyze the relation between atomic decompositions and Banach frames of the isomorphic coorbit spaces 
$\cH_{p,m}$ and $\widetilde{\cH}_{p,\widetilde{m}}$. 
Clearly, the set $\widetilde{\cB}_{\widetilde{w}}$ is just given by $\Psi(\cB_w)$.
Further, given a well-spread set $X := \{g_i : i\in\cI\}$ of $G$
we  can check that 
$\widetilde{X} = \iota(X) = \{\widetilde{g_i} = \iota(g_i) : i\in\cI \}$ 
is a well-spread set of $\widetilde{G}$.
Then, we obtain from the atomic decomposition 
\begin{equation*}
  f
=
  \sum_{i\in\cI}
  c_i(f)\pi(g_i)\psi
\end{equation*}
of $f \in \cH_{p,m}$ and the atomic decomposition of 
$\widetilde{f} = \Psi f$ by
\begin{equation*}
 \widetilde{f} 
= 
  \Psi f 
= 
  \Psi \left(\sum_{i\in\cI} c_i(f) \pi(g_i) \psi\right)
=
  \sum_{i\in\cI}
  c_i(f)
  \widetilde{\pi}(\widetilde{g}_i)\widetilde{\psi},
\end{equation*}
i.e., $\widetilde{f}$ can be decomposed using the same sequence of coefficients.
Concerning the Banach frame property of 
$\{\widetilde{\pi}(\widetilde{g}_i)\widetilde{\psi}: i \in \cI \}$  
it remains to deduce the reconstruction operators 
$\widetilde\cR\colon \ell_{p,\widetilde{m}} \rightarrow \widetilde{\cH}_{p,\widetilde{m}}$ 
given a reconstruction operator 
$\cR\colon\ell_{p,m} \rightarrow \cH_{p,m}$.
Let the sequence of moments of $\widetilde{f}$ be given as 
$
\{
  \langle 
    \widetilde{f},\widetilde{\pi}(g_i)\widetilde{\psi}
  \rangle_{
            (\widetilde{\cH}_{1,w})^{\sim} \times \widetilde{\cH}_{1,w}
          }
  \}
  \in\ell_{p,\widetilde{m}}
$. 
Using  $f = \Psi^*\widetilde{f}$ we get
\begin{equation*}
  \langle
    \widetilde{f},
    \widetilde{\pi}(\widetilde{g}_i)\widetilde{\psi}
  \rangle_{(\widetilde{\cH}_{1,w})^{\sim} \times \widetilde{\cH}_{1,w}}  
=
  \langle
    \widetilde{f},
    \Psi \left(\pi(g_i)\psi \right)
      \rangle_{(\widetilde{\cH}_{1,w})^{\sim} \times \widetilde{\cH}_{1,w}} 
=
  \langle 
    \Psi^* \widetilde{f},
    \pi(g_i)\psi 
  \rangle_{\cH_{1,w}^{\sim} \times \cH_{1,w}}
\end{equation*}
and since
$
  \cR(\{\langle \Psi^* \widetilde{f},\pi(g_i)\psi \rangle \}_{i\in\cI}) 
= 
  \Psi^*\widetilde{f}
$ 
that 
$\widetilde{\cR} = (\Psi^*)^{-1}\circ \cR$.
%
\subsection{Application to the connected (Toeplitz) shearlet groups and their isomorphic subgroups in the symplectic group}
In this section we want to use the results from the previous section to establish coorbit spaces 
for the groups $\TDS(d,\RR)$ and $\TDS_T(d,\RR)$ using the coorbit spaces for $\SS^{+}$ and $\SS^{+}_{T}$.
First, we have to determine the latter ones.
\paragraph{Square integrable representations of $\SS^{+}$ and $\SS^{+}_{T}$.}
The coorbit spaces for the full shearlet group $\SS$ and $\cH = L_2(\RR^d)$ 
were introduced in \cite{DKST09,DST10} and for the full Toeplitz shearlet group $\SS_T $ in \cite{DHT12}.
However, we cannot use these results directly since the respective representations are not irreducible if restricted to
the connected groups $\SS^{+}$ and $\SS^{+}_{T}$. 
Therefore, we consider instead of  $L_2(\RR^d)$
the Hilbert space 
\begin{equation*}
    L_2(\Theta_L) 
:= 
  \{ f\in L_2(\RR^d) : \supp \hat{f} \subseteq \Theta_L\},
\end{equation*}
where $\Theta_L$ denotes the halfspace
\begin{equation*} 
  \Theta_L 
:= 
  \{\xi\in\RR^d : \xi_1 \leq 0\}
\end{equation*}
and $\hat f = \cF f$ is the Fourier transform of $f$.

To shorten the notation we write in the following $\SS^{+}_{(T)}$ to address both $\SS^{+}$ and $\SS_{T}^{+}$.
Further, we just use $A$ for the dilations in \eqref{eq:dilationAndShearMatrices} and $S$ for the shears in \eqref{eq:shearmatrices}
for both groups and denote by $\mu$ their left Haar measures.
We define a representation $\pi\colon \SS^{+}_{(T)} \rightarrow \cU(L_2(\Theta_L) )$ by
\begin{equation} \label{eq:shearlet_representation}
  \pi(a,s,t)f(x) = f_{a,s,t}(x):= (\det(A))^{-\frac{1}{2}} f(A^{-1} S^{-1}(x-t))
\end{equation}
for all $(a,s,t)\in\SS_{(T)}^{+}$ and all $f \in L_2(\Theta_L)$.
The Fourier transform $\hat{f}_{a,s,t}$ of $f_{a,s,t}$ is given by
\begin{equation}\label{eq:represenatation_Fourier}
 \hat{f}_{a,s,t} (\omega)
:=
  \hat{\pi}(a,s,t)\hat{f}(\omega) 
= 
  (\det(A))^{\frac{1}{2}}
  \hat{f} (A^\tT S^\tT \omega) 
  e^{-2\pi i \langle t,\omega\rangle}.
\end{equation}
Note that the representations $\pi$ and $\hat{\pi}$ are equivalent in the sense that 
\begin{equation*}
  \cF\pi \cF^{-1}=\hat\pi.
\end{equation*}
Similarly as it is done for the full shearlet and shearlet Toeplitz group in \cite{DKMSST08,DKST09,DST10,DT10}, we can prove 
that $\pi$ is indeed a unitary representation of $\SS^{+}_{(T)}$ on $L_2(\Theta_L)$.

The following lemma shows that the unitary representation $\pi$ defined
in \eqref{eq:shearlet_representation} is also square integrable.
%
\begin{lemma} \label{theo:irreducible}
A function $\psi \in L_2(\Theta_L)$
is admissible if and only if it
fulfills the \emph{admissibility condition}
\begin{equation}\label{eq:admissibility_s}
  0
<
  C_\psi 
:=
  \int_{\Theta_L} \frac{\lvert\hat{\psi}(\omega)\rvert^2}{\lvert\omega_1\rvert^d} \, d  \omega  < \infty.
\end{equation}
Then, for any $f \in L_2(\Theta_L)$ the following equality holds true:
\begin{equation}\label{eq:square_Integrability}
  \int_{\SS^{+}_{(T)}} \lvert\langle f, \psi_{a,s,t} \rangle\rvert^2 \, d \mu (a,s,t)
=  
  C_\psi \, \lVert f\rVert_{L_2(\Theta_L)}^2.
\end{equation}
In particular, the unitary representation $\pi$ is irreducible and hence square integrable.
\end{lemma}
%
\begin{proof}
Observe that
\begin{equation} \label{eq:shearlet_transform_convolution}
  \cSH_f(a,s,t) 
:=
  \langle f, \psi_{a,s,t}\rangle
= 
  \langle f,\det(A)^{-\frac{1}{2}} \psi(A^{-1} S^{-1}(\,\cdot\, -t))\rangle
= 
  f * \psi^*_{a,s,0}(t),
\end{equation}
where $\psi^*_{a,s,t}:=\overline{\psi_{a,s,t}(-\,\cdot\, )}$.
For fixed $a$ and $s$, since $f,\psi_{a,s,0}\in L_2(\RR^d)$, then by standard facts 
$f*\psi_{a,s,0}=\cF^{-1}(\hat{f}\widehat{\psi^*}_{a,s,0})$ and by 
Plancherel's theorem we obtain
\[
  \int_{\RR^d} \lvert f * \psi^*_{a,s,0}(t)\rvert^2 dt
=
\int_{\widehat\RR^d} 
  \lvert \hat{f}(\omega)\rvert^2
  \lvert \widehat{\psi^*}_{a,s,0}(\omega)\rvert^2 d\omega, 
\]
where the left hand side is finite if and only if the right hand side is such. Hence, by Fubini, 
\eqref{eq:represenatation_Fourier}, and \eqref{eq:shearlet_transform_convolution} we obtain
\begin{align*}
  \int_{\SS^{+}_{(T)}} \lvert \langle f,\psi_{a,s,t}\rangle\rvert^2\, \frac{da}{a^{d+1}}\, ds\, dt 
&=
  \int_{\SS^{+}_{(T)}} \lvert f * \psi^*_{a,s,0}(t)\rvert^2 dt\,  ds\, \frac{da}{a^{d+1}} 
\\
&= 
  \int_{\RR_+} \int_{\RR^{d-1}} \int_{\Theta_L} 
  \lvert \hat{f}(\omega)\rvert^2
  \lvert \widehat{\psi^*}_{a,s,0}(\omega)\rvert^2 d\omega\,  ds\, \frac{da}{a^{d+1}} 
\\
&= 
  \int_{\RR_+} \int_{\RR^{d-1}} \int_{\Theta_L} 
  \lvert\hat{f}(\omega)\rvert^2 \frac{\det(A)}{a^{d+1}} \lvert\hat{\psi}(S^\tT A^\tT \omega)\rvert^2 d\omega\, ds\,  da 
\\
&= 
  \int_{\Theta_L} \int_{\RR_+} \int_{\RR^{d-1}}
  \lvert\hat{f}(\omega)\rvert^2 
  \frac{\det(A)}{a^{d+1}} 
  \lvert\hat{\psi}(S^\tT A^\tT \omega)\rvert^2 \, ds\,  da \, d\omega.
\end{align*}
To continue  we need the concrete matrices $A$ and $S$ for the connected shearlet and Toeplitz shearlet groups. 
We restrict our attention to $\SS^+_T$. The conclusions for $\SS^+$ can be drawn in the same way
with slightly simpler substitutions.
Then
\begin{align*}
  T_{s}^\tT A_a^\tT \omega 
&=
  \begin{pmatrix} 
    1       & 0       &        & \cdots & 0      \\
    s_1     & 1       & 0      &        & \vdots \\
    s_2     & s_1     & 1      & 0      &        \\
    \vdots  & \ddots  & \ddots & \ddots & 0      \\
    s_{d-1} & s_{d-2} & \ldots & s_1    & 1
  \end{pmatrix} 
  \begin{pmatrix}
    a\omega_1 \\ 
    a\omega_2 \\
    \vdots    \\
    \vdots    \\
    a\omega_d
 \end{pmatrix}
=
  \left(
  \begin{array}{r}
    a\omega_1 \\
    s_1 a\omega_1 + a\omega_2 \\
    \vdots\\
    \vdots\\
    s_{d-1} a\omega_1 + a s_{d-2}\omega_2 + \ldots + a\omega_d
  \end{array}
  \right)
\end{align*}
and substituting 
$\xi_d := s_{d-1} a\omega_1 + \ldots + a\omega_d$, 
$\xi_{d-1} := s_{d-2} a\omega_1 + \ldots + a\omega_{d-2}$, $\ldots$ ,
$\xi_2 := s_1 a \omega_1 + a\omega_2$
we obtain with $d\widetilde{\xi}:=d \xi_2\ldots d\xi_d$ and $d\widetilde{\xi} = (a \lvert\omega_1\rvert)^{d-1} \, ds$ the equality
\begin{align*}
  \int_{\SS^{+}_{(T)}} \lvert\langle f,\psi_{a,s,t}\rangle\rvert^2\, \frac{da}{a^{d+1}}\, ds\, dt
&= 
  \int_{\Theta_L} \int_{\RR_+} \int_{\RR^{d-1}}  
  \lvert\hat{f}(\omega)\rvert^2 a^{-d}\lvert\omega_1\rvert^{-(d-1)}
  \lvert\hat{\psi}(a \omega_1, \xi_2, \ldots\, ,\xi_d )\rvert^2 
  \, d\widetilde{\xi}\,  da \, d\omega
\\
&= 
  \int_{\Theta_L} \int_{\RR_-} \int_{\RR^{d-1}}  
  \lvert\hat{f}(\omega)\rvert^2 \lvert\xi_1\rvert^{-d}
  \lvert\hat{\psi}(\xi_1, \xi_2, \ldots\, ,\xi_d )\rvert^2 
  \, d\widetilde{\xi}\,  d\xi_1 \, d\omega
\\
&=
  \int_{\Theta_L} \lvert\hat{f}(\omega)\rvert^2 d\omega
  \int_{\Theta_L} \frac{\lvert\hat{\psi}(\xi)\rvert^2 }{\lvert\xi_1\rvert^{d}} \, d\xi
=
  \lVert f\rVert_{L_2(\Theta_L)}^2 \int_{\Theta_L} \frac{\lvert\hat{\psi}(\xi )\rvert^2 }{\lvert\xi_1\rvert^{d}} \, d\xi
\end{align*}
where again the left hand side is finite if and only if the right hand side is such.
This yields \eqref{eq:square_Integrability} and since there exist functions $\psi \in L_2(\Theta_L)$ for which \eqref{eq:admissibility_s}
is finite we have verified \eqref{eq:admissibility}.
Next, we show how \eqref{eq:square_Integrability} implies the irreducibility of $\pi$. 
Suppose by contradiction that it is not irreducible. Then there exist two non-zero functions 
$f,\psi\in L_2(\RR^d)$ for which $ \langle f,\psi_{a,s,t}\rangle=0$ as a function of $(a,s,t)$. But then the previous string of equalities yields that the right hand side vanishes. But this in turn implies that either
$f=0$ or $\hat\psi=0$, a contradiction.
\end{proof}
The whole coorbit space setting  worked out in  \cite{DKST09,DST10,DHT12} for the full  shearlet  and Toeplitz shearlet groups 
can now be modified in a straightforward way to the connected groups.
\paragraph{Equivalent representations of $\TDS(d,\RR)$ and $\TDS_T(d,\RR)$.}
We want to exploit the isomorphisms $\kappa^+_{(T)}$ between $\SS_{(T)}^{+}$ and the subgroups $\TDS_{(T)}(d,\RR)$ 
of the symplectic group to define coorbit spaces for the latter subgroups with respect to their metaplectic representations. 
For this purpose, we define the diffeomorphism
\begin{equation*}
  Q \colon \Theta_L \to \Theta_L
\quad\text{with}\quad
  \xi
\mapsto 
  Q(\xi)
:= 
  -\frac{1}{2}(\xi^2_1,\xi_1\xi_2,\ldots,\xi_1\xi_{d})^\tT.
\end{equation*}
The inverse of $Q$ is given for $\xi_1 < 0$ by 
$
  Q^{-1}(\xi) 
= 
  \sqrt{2}(-\sqrt{-\xi_1},\frac{\xi_2}{\sqrt{-\xi_1}},\ldots,\frac{\xi_d}{\sqrt{-\xi_1}})^\tT
$.
Further, the absolute value of the determinants of the Jacobian $\cJ_Q$ of $Q$ and its inverse read
\begin{equation} \label{determinants}
    \left\lvert \det(\cJ_Q(\xi)) \right\rvert
  =
    2^{1-d} \lvert\xi_1\rvert^d
    \quad\text{and}\quad
    \left\lvert \det(\cJ_{Q^{-1}}(\xi)) \right\rvert
  =
    (\sqrt{2})^{d-2}\lvert\xi_1\rvert^{-\frac{d}{2}}
    .
  \end{equation}
Based on $Q$ we can define the isomorphism 
$\Psi\colon L_2(\Theta_L)\to L_2(\Theta_L)$ as $\Psi = \cF^{-1} \hat \Psi \cF$, where
\begin{equation} \label{def_Psi}
  \hat \Psi \hat f (\xi) 
= 
  \lvert\det(\cJ_{Q^{-1}}(\xi) )\rvert^{\frac{1}{2}} \hat f(Q^{-1}(\xi)).
\end{equation}
Its inverse is given by
\begin{equation*}
  \hat \Psi^{-1}F(\xi) 
= 
  \lvert\det(\cJ_{Q}(\xi))\rvert^{\frac{1}{2}}F(Q(\xi)).
\end{equation*}

The semi-direct product $\Sigma \rtimes H$ in \eqref{general_metaplectic} possesses a (mock) 
\emph{metaplectic representation} $\pi^m$ on $L_2(\Theta_L)$
defined for all $g \in \Sigma\rtimes H$ by
\begin{equation*}
  \hat \pi^m ( g) \hat f (\xi) 
:= 
  (\det M)^{-\frac{1}{2}} e^{-2\pi i\langle t,Q(\xi) \rangle} \hat f( M^{-1} \xi), \qquad f \in L_2(\Theta_L).
\end{equation*}
For $\TDS(d,\RR)$  this metaplectic representation becomes
\begin{equation} \label{meta_TDS}
  \hat{\pi}^m(a,s,t)\hat{f}(\xi)
=
  a^{\frac{1}{2} - \frac{d}{4} + \frac{\gamma}{2}(d-1)}
  e^{-2\pi i\langle t,Q(\xi) \rangle}
  \hat{f}(\widetilde{A}_{a,\gamma}^{-1} \widetilde{S}_s^{-1}\xi)
\end{equation}
with the matrices $\widetilde{A}_{a,\gamma}$ and $\widetilde{S}_s$ in \eqref{tilde_a_s},
and for 
$\TDS_T(d,\RR)$ just
\begin{equation} \label{meta_TDST}
  \hat{\pi}^m(a,s,t) \hat{f}(\xi) 
= 
  a^{\frac{d}{4}}
  e^{-2\pi i \langle t, Q(\xi)\rangle}
  \hat{f}(\sqrt{a}\, T_s^\tT \xi).
\end{equation}
We now show that these representations are equivalent to the representations \eqref{eq:shearlet_representation}
of the connected (Toeplitz) shearlet groups.
%
\begin{lemma} \label{iso_s_meta}
Let the isomorphism $\kappa^+\colon \SS^+ \rightarrow \TDS(d,\RR)$ be given by \eqref{eq:definition_g_+}.
Then the representation $\pi^m$ of $\TDS(d,\RR)$ in \eqref{meta_TDS} is equivalent 
to the representation $\pi$ of $\SS^+$ defined by \eqref{eq:shearlet_representation}
in the sense
\begin{equation}\label{eq:equivalence_regular}
  \hat \Psi\hat{\pi}^m(\kappa^{+}(\,\cdot\,)) \hat \Psi^{-1} 
= 
  \cF \pi \cF^{-1} 
= 
  \hat\pi.
\end{equation}
\end{lemma}
%
\begin{proof}
 First we verify that
  \begin{align} \nonumber
    Q(\widetilde{A}_{a,\gamma}^{-1}\widetilde{S}_s^{-1} \xi)
  &=
    Q
    \left(
      \begin{pmatrix}
        a^{\frac{1}{2}} & 0 \\ s a^{-\frac{1}{2}+\gamma} & a^{-\frac{1}{2} + \gamma}I_{d-1}
      \end{pmatrix}
      \begin{pmatrix} \xi_1 \\ \widetilde{\xi} \end{pmatrix}
    \right)
    \\
  &=
    Q(a^{\frac{1}{2}} \xi_1, s \xi_1 a^{-\frac{1}{2}+\gamma} + a^{-\frac{1}{2}+\gamma} \widetilde{\xi})
  =
    -\frac{1}{2}
    \left(a \xi_1^2, s a^\gamma \xi_1^2 + a^\gamma \xi_1 \widetilde{\xi}\right) \nonumber
    \\
    &=
      -\frac{1}{2}
    A_{a,\gamma}
    \begin{pmatrix}
      \xi_1^2 \\ s \xi_1^2 +  \xi_1 \widetilde{\xi}
    \end{pmatrix} \nonumber
    \\
  &=
    A_{a,\gamma} S_s^\tT Q(\xi) \label{eq:properties_Q_regular}.
  \end{align}
  Then we conclude by \eqref{def_Psi}, the definition of  $\hat \pi_m$ and  \eqref{eq:properties_Q_regular} that
  \begin{align*}
  &
    (\hat \Psi \hat{\pi}^m(a,s,t)\hat \Psi^{-1} \hat{\psi})(\xi)
  \\
  &=
    \lvert \det(\cJ_{Q^{-1}}(\xi))\rvert^{\frac{1}{2}}
    \hat{\pi}^m(a,s,t)(\hat \Psi^{-1} \hat{\psi})(Q^{-1}(\xi))
  \\
  &=
    \lvert \det(\cJ_{Q^{-1}}(\xi))\rvert^{\frac{1}{2}} 
    a^{\frac{1}{2} - \frac{d}{4} + \frac{\gamma}{2}(d-1)}
    (\hat \Psi^{-1} \hat{\psi})(\widetilde{A}_{a,\gamma}^{-1}\widetilde{S}_s^{-1} Q^{-1}(\xi))
    e^{-2\pi i \langle t,Q(Q^{-1}(\xi)) \rangle}
  \\
  &=
    \lvert \det(\cJ_{Q^{-1}}(\xi))\rvert^{\frac{1}{2}} 
    a^{\frac{1}{2} - \frac{d}{4} + \frac{\gamma}{2}(d-1)}
    \lvert \det(\cJ_Q(\widetilde{A}_{a,\gamma}^{-1}\widetilde{S}_s^{-1} Q^{-1}(\xi)))\rvert^{\frac{1}{2}}
    \hat{\psi}(Q(\widetilde{A}_{a,\gamma}^{-1}\widetilde{S}_s^{-1} Q^{-1}(\xi)))
    e^{-2\pi i \langle t,\xi \rangle}
  \\
  &=
    \lvert \det(\cJ_{Q^{-1}}(\xi))\rvert^{\frac{1}{2}} 
    a^{\frac{1}{2} - \frac{d}{4} + \frac{\gamma}{2}(d-1)}
    \lvert \det(\cJ_Q(\widetilde{A}_{a,\gamma}^{-1}\widetilde{S}_s^{-1} Q^{-1}(\xi)))\rvert^{\frac{1}{2}}
    \hat{\psi}(A_{a,\gamma} S_s^\tT \xi)
    e^{-2\pi i \langle t,\xi \rangle}.
  \end{align*}
By \eqref{determinants} we have
\begin{equation} \label{one_det}
  \lvert \det(\cJ_{Q^{-1}}(\xi)) \rvert^{\frac{1}{2}}
=
  (\sqrt{2})^{\frac{d-2}{2}}(\sqrt{-\xi_1})^{-\frac{d}{2}}.
\end{equation}
Simplifying
\begin{align*}
  \lvert \det(\cJ_{Q}(\widetilde{A}_{a,\gamma}^{-1}\widetilde{S}_s^{-1} Q^{-1}(\xi))) \rvert^{\frac{1}{2}}
&=
  \left\lvert 
  \det\left(\cJ_{Q}
  \left(
  \begin{pmatrix}
    a^{\frac{1}{2}}           & 0                         \\
    s a^{-\frac{1}{2}+\gamma} & a^{-\frac{1}{2} + \gamma}I_{d-1}
  \end{pmatrix}
  \begin{pmatrix}
    -\sqrt{-2\xi_1}\\
    \frac{\sqrt{2}\xi_2}{\sqrt{-\xi_1}} \\
    \vdots \\
    \frac{\sqrt{2}\xi_d}{\sqrt{-\xi_1}}
  \end{pmatrix}
  \right)
  \right)
  \right\rvert^{\frac{1}{2}}
\\
&=
  (
  2^{1-d}
  (\sqrt{-2a\xi_1})^{d}
  )^{\frac{1}{2}}
=
  (\sqrt{2})^{(1-d)+\frac{d}{2}}
  a^{\frac{d}{4}}
  (\sqrt{-\xi_1})^\frac{d}{2}
\end{align*}
we obtain further
\begin{align*}
  \lvert \det(\cJ_{Q^{-1}}(\xi)) \rvert^{\frac{1}{2}}
  \lvert \det(\cJ_{Q}(\widetilde{A}_{a,\gamma}^{-1}\widetilde{S}_s^{-1} Q^{-1}(\xi))) \rvert^{\frac{1}{2}}
&=
  (\sqrt{2})^{\frac{d-2}{2}}(\sqrt{-\xi_1})^{-\frac{d}{2}}
  (\sqrt{2})^{(1-d)+\frac{d}{2}}
  a^{\frac{d}{4}}
  (\sqrt{-\xi_1})^\frac{d}{2}
\\
&=
  a^{\frac{d}{4}}
\end{align*}
and finally
\begin{align*}
  (\hat \Psi \hat{\pi}^m(a,s,t)\hat \Psi^{-1} \hat{\psi})(\xi)
&=
  a^{\frac{1}{2} + \frac{\gamma}{2}(d-1)}
  \hat{\psi}(A_{a,\gamma} S_s^\tT \xi)
  e^{-2\pi i \langle t,\xi \rangle}
\\
&=
  \lvert \det(A_{a,\gamma})\rvert^\frac{1}{2}
  \hat{\psi}(A_{a,\gamma} S_s^\tT \xi)
  e^{-2\pi i \langle t,\xi \rangle}
\\
&=
  \hat{\pi}(a,s,t)\hat{\psi}(\xi).
  \qedhere
\end{align*}
\end{proof}

Equation \eqref{eq:equivalence_regular} can be illustrated by the following diagram.
\begin{center}
\begin{tikzpicture}[]
  \node (UL2a) {$\mathcal{U}(L_2(\Theta_L))$};
  \node (UL2b) [right = 3.5cm of UL2a] {$\mathcal{U}(L_2(\Theta_L))$};
  \node (S) [below =2cm of UL2a] {$\SS^+$};
  \node (G) [below =2cm of UL2b] {$\TDS(d,\RR)$};
  \draw[->] (S) to node[shape=rectangle,fill=white,fill opacity=1,inner sep=5pt] {$\hat{\pi}$} (UL2a);
  \draw[->] (UL2a) to node[shape=rectangle,fill=white,fill opacity=1,inner sep=5pt] {$\hat \Psi^{-1}(\,\cdot\,)\hat \Psi$} (UL2b);
  \draw[->] (S) to node[shape=rectangle,fill=white,fill opacity=1,inner sep=5pt] {$\kappa^{+}$} (G);
  \draw[->] (G) to node[shape=rectangle,fill=white,fill opacity=1,inner sep=5pt] {$\hat{\pi}^m$} (UL2b);
\end{tikzpicture}
\end{center}
%

%
\begin{lemma}
Let the isomorphism $\kappa_T^+\colon \SS_T^+ \rightarrow \TDS_T(d,\RR)$ be given by \eqref{eq:definition_g_+_T}.
Then the representations $\pi^m$ of $\TDS_T(d,\RR)$ in \eqref{meta_TDST} and $\pi$ of $\SS_T^+$ 
given by \eqref{eq:shearlet_representation}
are equivalent in the sense that
\begin{equation}\label{eq:equivalence_toeplitz}
  \hat \Psi\hat{\pi}^m(\kappa^{+}_{T}(\,\cdot\,))\hat \Psi^{-1} = \cF \pi \cF^{-1} = \hat\pi.
\end{equation} 
\end{lemma}
\begin{proof}
  By definition of $\sigma(t)$ in \eqref{def:spec_sigma} we have
  \begin{align*}
    \langle \sigma(t_1,\tilde t) \xi, \xi \rangle
  &=
    \left\langle 
      \begin{pmatrix} t_1 & \frac{1}{2}\tilde{t}^\tT \\ \frac{1}{2}\tilde{t} & 0 \end{pmatrix}
      \begin{pmatrix} \xi_1 \\ \widetilde{\xi} \end{pmatrix}
    ,
      \begin{pmatrix} \xi_1 \\ \widetilde{\xi} \end{pmatrix}
    \right\rangle
  =
    \left\langle 
      \begin{pmatrix} t_1 \xi_1 + \frac{1}{2}\tilde{t}^\tT\widetilde{\xi} \\ \frac{1}{2}\xi_1\tilde{t}\end{pmatrix}
    ,
      \begin{pmatrix} \xi_1 \\ \widetilde{\xi} \end{pmatrix}
    \right\rangle
  \\
  &=
    t_1 \xi_1^2 + \frac{1}{2}\xi_1\tilde{t}^\tT\widetilde{\xi} + \frac{1}{2} \xi_1\tilde{t}^\tT\widetilde{\xi}
  =
    t_1 \xi_1^2 + \xi_1\tilde{t}^\tT\widetilde{\xi}
  \\
   &=
    -2\langle t,Q(\xi) \rangle.
  \end{align*}
  Using \eqref{eq:toeplitz_conjugation} we obtain for all $t \in \RR^d$ that
  \begin{equation*}
    -2\langle t,Q(T_s^\tT \xi) \rangle
  =
    \langle \sigma(t_1,\tilde{t}) T_s^\tT \xi,T_s^\tT \xi \rangle
  =
    \langle T_s \sigma(t_1,\tilde{t}) T_s^\tT \xi, \xi \rangle
  =
    \langle \sigma(T_s t) \xi,\xi \rangle
     \end{equation*}
    and by the above relation
    \begin{equation*}
    \langle \sigma(T_s t) \xi,\xi \rangle
  =
    -2 \langle T_s t, Q(\xi) \rangle
  =
    -2 \langle t,T_s^\tT Q(\xi) \rangle
  \end{equation*}
  such that $Q(T_s^\tT \xi) = T_s^\tT Q(\xi)$ and by definition of $Q$ further
  \begin{equation}\label{eq:properties_Q}
    Q(\sqrt{a} T_s^\tT \xi)
  =
    a T_s^\tT Q( \xi).
  \end{equation}
  Now we can compute
  \begin{align*}
  &
    (\hat \Psi \hat{\pi}^m(a,s,t)\hat \Psi^{-1} \hat{\psi})(\xi)
  \\
  &=
    \lvert \det(\cJ_{Q^{-1}}(\xi))\rvert^{\frac{1}{2}}\hat{\pi}^m(a,s,t)(\hat \Psi^{-1} \hat{\psi})(Q^{-1}(\xi))
  \\
  &=
    \lvert \det(\cJ_{Q^{-1}}(\xi))\rvert^{\frac{1}{2}} 
    a^{\frac{d}{4}} 
    (\hat \Psi^{-1} \hat{\psi})(\sqrt{a}T_s^\tT Q^{-1}(\xi))
    e^{-2\pi i \langle t,Q(Q^{-1}(\xi)) \rangle}
  \\
  &=
    \lvert \det(\cJ_{Q^{-1}}(\xi))\rvert^{\frac{1}{2}} 
    a^{\frac{d}{4}} 
    \lvert \det(\cJ_Q(\sqrt{a}T_s^\tT Q^{-1}(\xi)))\rvert^{\frac{1}{2}}
    \hat{\psi}(Q(\sqrt{a}T_s^\tT Q^{-1}(\xi)))
    e^{-2\pi i \langle t,\xi \rangle}
  \\
  &=
    \lvert \det(\cJ_{Q^{-1}}(\xi))\rvert^{\frac{1}{2}} 
    a^{\frac{d}{4}} 
    \lvert \det(\cJ_Q(\sqrt{a}T_s^\tT Q^{-1}(\xi)))\rvert^{\frac{1}{2}}
    \hat{\psi}(a T_s^\tT \xi)
    e^{-2\pi i \langle t,\xi \rangle}.
  \end{align*}
Applying
\begin{align*}
  \lvert \det(\cJ_{Q}(\sqrt{a}T_s^\tT Q^{-1}(\xi))) \rvert^{\frac{1}{2}}
&=
  \left\lvert 
  \det\left(\cJ_{Q}
  \left(\sqrt{a}
  \begin{pmatrix}
    1       &   &        &   \\
    s_1     & 1 &        &   \\
    \vdots  &   & \ddots &   \\
    s_{d-1} &   &        & 1
  \end{pmatrix}
  \begin{pmatrix}
    -\sqrt{-2\xi_1}\\
    \frac{\sqrt{2}\xi_2}{\sqrt{-\xi_1}} \\
    \vdots \\
    \frac{\sqrt{2}\xi_d}{\sqrt{-\xi_1}}
  \end{pmatrix}
  \right)
  \right)
  \right\rvert^{\frac{1}{2}}
\\
&=
  \left\lvert 
  \det
  \left(\cJ_{Q}
  \left(
  -\sqrt{-2a \xi_1},\ldots
  \right)
  \right)
  \right\rvert^{\frac{1}{2}}
=
  (
  2^{1-d}
  (\sqrt{-2a \xi_1})^{d}
  )^{\frac{1}{2}}
\\
&=
  (\sqrt{2})^{(1-d)+\frac{d}{2}}
  a^{\frac{d}{4}}
  (\sqrt{-\xi_1})^\frac{d}{2}
\end{align*}
and \eqref{one_det} we get 
$
  \lvert \det(\cJ_{Q^{-1}}(\xi)) \rvert^{\frac{1}{2}}
  \lvert \det(\cJ_{Q}(\sqrt{a}S_s^\tT Q^{-1}(\xi))) \rvert^{\frac{1}{2}}
=
  a^{\frac{d}{4}}
$ 
and finally
\begin{equation*}
  (\hat \Psi \hat{\pi}^m(a,s,t)\hat \Psi^{-1} \hat{\psi})(\xi)
=
  a^{\frac{d}{2}} 
  \hat{\psi}(a T_s^\tT \xi)
  e^{-2\pi i \langle t,\xi \rangle}
=
  \hat{\pi}(a,s,t)\hat{\psi}(\xi).
\end{equation*}
\end{proof}
Observe that, in light of \eqref{eq:equivalence_toeplitz}, a vector $\hat{\psi}\in L_2(\Theta_L)$, 
is admissible for $\hat\pi^m$ if and only if 
\begin{equation*}
  \int_{\Theta_L}
  \frac{\lvert \hat \Psi \hat{\psi}(\xi)\rvert^2}{\lvert\xi_1\rvert^d}
  \, d\xi
<
  +\infty.
\end{equation*}

%
\paragraph{Metaplectic coorbit spaces.}
Based on the equivalence of the metaplectic representations of the subgroups $\TDS_{(T)}(d,\RR)$ to square integrable 
representations of $\SS^+_{(T)}$ we can apply the results from Subsection~\ref{subsec:iso}
to define coorbit spaces with respect to $\TDS_{(T)}(d,\RR)$.

Let $\psi\in L_2(\Theta_L)$ be an admissible shearlet.
Then the transform
$\cSH_\psi\colon L_2(\Theta_L)\to L_2(\SS^+_{(T)})$ defined by
\begin{equation*}
  \cSH_\psi (f) (a,s,t)
=
  \langle f, \pi(a,s,t)\psi\rangle_{L_2(\Theta_L)}
\end{equation*}
is called the \emph{continuous (Toeplitz) shearlet transform}, whereas 
the transform
$\cSH^m_\psi\colon L_2(\Theta_L)\to L_2(\TDS_{(T)}(d))$ defined by
\begin{equation*}
  \cSH^m_\psi f (\kappa^{+}_{(T)}(a,s,t))
=
  \langle f, \pi^m(\kappa^{+}_{(T)}(a,s,t))\psi\rangle_{L_2(\Theta_L)}
\end{equation*}
is called the \emph{metaplectic continuous shearlet transform}. 

Using the above equivalences the 
\emph{(Toeplitz) shearlet coorbit spaces}
\begin{equation*}
  \cSC_{p,m}
:=
  \{ 
    f\in \cH_{1,w}^\sim 
  : 
    \cSH_\psi(f)\in L_{p,m}(\SS^+_{(T)})
  \}
\end{equation*}
and the \emph{metaplectic shearlet coorbit spaces}
\begin{equation*}
  \cSC^m_{p,m}
:=
  \{
    f\in (\widetilde{\cH}_{1,\widetilde{w}})^\sim 
  : 
    \cSH^m_\psi(f)\in L_{p,\widetilde{m}}(\TDS_{(T)}(d))
  \}
\end{equation*}
are diffeomorphic.

\section*{Acknowledgements}
This work has been supported by Deutsche Forschungsgemeinschaft (DFG), Grants DA 360/19--1 and STE 571/11--1.  
Some parts of the paper have been written during a stay at the Erwin-Schrödinger Institute (ESI), Vienna,  Workshop on ``Time-Frequency Analysis'', January 13-17 2014.  Therefore the support of ESI is also acknowledged.

F. De Mari and E.~De Vito  were partially supported by Progetto PRIN 2010-2011 
``Variet\`a reali e complesse: geometria, topologia e analisi armonica''.
They are members of the Gruppo Nazionale per l’Analisi Matematica, 
la Probabilità e le loro Applicazioni (GNAMPA)
of the Istituto Nazionale di Alta Matematica (INdAM).

S. Dahlke, S. Häuser, G. Steidl and G. Teschke were partially supported by DAAD Project 57056121, 
”Hochschuldialog mit Südeuropa 2013”.

\appendix
\section*{Appendix}
In the following we list the Lie brackets of the canonical matrices $D \in \mathcal{N}$ with the basis matrices $X_\nu$, $\nu \in \triangle$,
and $H_{1,0}$, $H_{0,1}$ from the root space decomposition of $\sptwor$
and the matrices $M_{\Gamma}$ defined by \eqref{matrix_M}.
\\[1ex]
\textbf{Case 1.} 
For 
$
D_1 =   a_1H_{1,0}   +   a_2H_{0,1}
$
we obtain
\begin{equation*}
\begin{array}{lcllcl}
  [D_1,X_{\alpha}]&=& (a_1-a_2)X_{\alpha}, &
	[D_1,X_{-\alpha}] &=& -(a_1-a_2)X_{-\alpha},
	\\
  {[}D_1,X_{\beta}]&=&  2a_2X_{\beta},&
	[D_1,X_{-\beta}]&=& - 2a_2X_{-\beta},
	\\
  {[}D_1,X_{\alpha+\beta}]&=&  (a_1+a_2)X_{\alpha+\beta}, & 
	[D_1,X_{-\alpha-\beta}]&=&  -(a_1+a_2)X_{-\alpha-\beta},
	\\
  {[}D_1,X_{2\alpha+\beta}]&=&  2a_1X_{2\alpha+\beta},&
  [D_1,X_{-2\alpha-\beta}]&=&  -2a_1X_{-2\alpha-\beta},
	\\
  {[}D_1,H_{1,0}]&=&  0,&
  [D_1,H_{0,1}]&=&   0
\end{array}
\end{equation*}
and
$M_\Gamma$ is a diagonal matrix with entries
{\small
\begin{equation*}
  \left(
    a_1 - a_2 - \Gamma,
    2a_2 - \Gamma,
    a_1 + a_2 - \Gamma,
    2a_1 - \Gamma,
    a_2 - a_1 - \Gamma,
    - 2a_2 - \Gamma,
    - a_2 - a_1 - \Gamma,
    - 2a_1 - \Gamma,
    - \Gamma,
    - \Gamma
  \right),
\end{equation*}}
where $a_1 \geq a_2 \geq 0$.

\textbf{Case 2.}
For
$
D_2 =  X_{-\alpha}
  +
  aH_{1,0}
  +
  aH_{0,1}
	$
we obtain
\begin{equation*}
\begin{array}{lcllcl}
  [D_2,X_{\alpha}]
&=&
  H_{1,0}-H_{0,1},
&
 [D_2,X_{-\alpha}]
&=&
  0,
\\
  {[}D_2,X_{\beta}]
&=& 
  2aX_{\beta},
&
	[D_2,X_{-\beta}]
&=&
  -2aX_{-\beta}+X_{-\alpha-\beta},
\\
  {[}D_2,X_{\alpha+\beta}]
&=&
  -2X_{\beta}+2aX_{\alpha+\beta},
&
	 [D_2,X_{-\alpha-\beta}]
&=&
  -2aX_{-\alpha-\beta}+X_{-2\alpha-\beta},
\\
  {[}D_2,X_{2\alpha+\beta}]
&=&
  -2X_{\alpha+\beta}+2aX_{2\alpha+\beta},
& 
  [D_2,X_{-2\alpha-\beta}]
&=&
  -2aX_{-2\alpha-\beta},
\\
  {[}D_2,H_{1,0}]
&=&
  X_{-\alpha},
&
  [D_2,H_{0,1}]
&=&
  -X_{-\alpha}
\end{array}
\end{equation*}
and 
{\small
\begin{equation*}
  M_\Gamma
=
  \begin{pmatrix}
    -\Gamma & 0           &  0           &  0           &  0      &  0           &  0           &  0           &  0      &  0      \\
     0      & 2a - \Gamma & -2           &  0           &  0      &  0           &  0           &  0           &  0      &  0      \\
     0      & 0           &  2a - \Gamma & -2           &  0      &  0           &  0           &  0           &  0      &  0      \\
     0      & 0           &  0           &  2a - \Gamma &  0      &  0           &  0           &  0           &  0      &  0      \\
     0      & 0           &  0           &  0           & -\Gamma &  0           &  0           &  0           &  1      & -1      \\
     0      & 0           &  0           &  0           &  0      & -2a - \Gamma &  0           &  0           &  0      &  0      \\
     0      & 0           &  0           &  0           &  0      &  1           & -2a - \Gamma &  0           &  0      &  0      \\
     0      & 0           &  0           &  0           &  0      &  0           &  1           & -2a - \Gamma &  0      &  0      \\
     1      & 0           &  0           &  0           &  0      &  0           &  0           &  0           & -\Gamma &  0      \\
    -1      & 0           &  0           &  0           &  0      &  0           &  0           &  0           &  0      & -\Gamma
  \end{pmatrix}
\end{equation*}}\noindent
with
\begin{equation*}
  \det \, M_\Gamma
=
  \Gamma^4 
  (\Gamma - 2a)^3 
  (\Gamma + 2a)^3, \qquad a \ge 0.
\end{equation*}
\textbf{Case 3.} 
For
$
  D_3
=
  bX_{\alpha}
  +
  bX_{-\alpha}
  +
  aH_{1,0}
  +
  aH_{0,1}
	$
we obtain
\begin{equation*}
\begin{array}{lcllcl}
  [D_3,X_{\alpha}]
&=&
  bH_{1,0}-bH_{0,1},
&
[D_3,X_{-\alpha}]
&=&
  -bH_{1,0}+bH_{0,1},
\\
  {[}D_3,X_{\beta}]
&=&
  2aX_{\beta}+bX_{\alpha+\beta},
&
[D_3,X_{-\beta}]
&=&
  -2aX_{-\beta} + bX_{-\alpha-\beta},
\\
  {[}D_3,X_{\alpha+\beta}]
&=&
  -2bX_{\beta}+2aX_{\alpha+\beta}+bX_{2\alpha+\beta},
&
[D_3,X_{-\alpha-\beta}]
&=&
  -2bX_{-\beta}-2aX_{-\alpha-\beta} \\
	& & & & &+ bX_{-2\alpha-\beta},
\\
  {[}D_3,X_{2\alpha+\beta}]
&=&
  -2bX_{\alpha+\beta}+2aX_{2\alpha+\beta},
&
  [D_3,X_{-2\alpha-\beta}]
&=&
  -2bX_{-\alpha-\beta} -2aX_{-2\alpha-\beta},
\\
  {[}D_3,H_{1,0}]
&=&
  -b X_{\alpha}+bX_{-\alpha},
&
  [D_3,H_{0,1}]
&=&
  b X_{\alpha}-bX_{-\alpha}
\end{array}
\end{equation*}
and
{\small
\begin{equation*}
  M_\Gamma
=
  \begin{pmatrix}
    -\Gamma & 0           &  0           &  0           &  0      &  0           &  0           &  0           & -b      &  b      \\
     0      & 2a - \Gamma & -2b          &  0           &  0      &  0           &  0           &  0           &  0      &  0      \\
     0      & b           &  2a - \Gamma & -2b          &  0      &  0           &  0           &  0           &  0      &  0      \\
     0      & 0           &  b           &  2a - \Gamma &  0      &  0           &  0           &  0           &  0      &  0      \\
     0      & 0           &  0           &  0           & -\Gamma &  0           &  0           &  0           &  b      & -b      \\
     0      & 0           &  0           &  0           &  0      & -2a - \Gamma & -2b          &  0           &  0      &  0      \\
     0      & 0           &  0           &  0           &  0      &  b           & -2a - \Gamma & -2b          &  0      &  0      \\
     0      & 0           &  0           &  0           &  0      &  0           &  b           & -2a - \Gamma &  0      &  0      \\
     b      & 0           &  0           &  0           & -b      &  0           &  0           &  0           & -\Gamma &  0      \\
    -b      & 0           &  0           &  0           &  b      &  0           &  0           &  0           &  0      & -\Gamma 
  \end{pmatrix}
\end{equation*}}
with
\begin{equation*}
  \det \, M_\Gamma
=
  \Gamma^2 
  (\Gamma^2 + 4 b^2)     
  (\Gamma - 2a)
  (\Gamma + 2a)
  \left((\Gamma - 2 a)^2 + 4 b^2 \right)
  \left((\Gamma + +2 a)^2 + 4 b^2 \right), \quad a,b>0.
\end{equation*}
\textbf{Case 4.} 
For
$
  D_4
=
  \eps X_{\alpha+\beta}
  -
  X_{-\alpha}
  -
  \frac{\eps}{2}X_{-2\alpha-\beta}
$ 
we obtain
\begin{equation*}
\begin{array}{lcllcl}
  [D_4,X_{\alpha}]
&=&
  -\eps X_{2\alpha+\beta}
  -\eps X_{-\alpha-\beta}
  -H_{1,0}
  +H_{0,1},
&
	 [D_4,X_{-\alpha}]
&=&
  2\eps X_{\beta},
\\
  {[}D_4,X_{\beta}]
&=&
  0,
&
[D_4,X_{-\beta}]
&=&
  -\eps X_{\alpha}
  -X_{-\alpha-\beta},
\\
  {[}D_4,X_{\alpha+\beta}]
&=&
  2X_{\beta}
  +\eps X_{-\alpha},
&
[D_4,X_{-\alpha-\beta}]
&=&
  -X_{-2\alpha-\beta}
  -\eps H_{1,0}
\\
&&&&&
  -\eps H_{0,1},
\\
  {[}D_4,X_{2\alpha+\beta}]
&=&
  2X_{\alpha+\beta}
  -2\eps H_{1,0},
&
  [D_4,X_{-2\alpha-\beta}]
&=&
  2\eps X_{-\alpha},
\\
  {[}D_4,H_{1,0}]
&=&
  -\eps X_{\alpha+\beta}
  -X_{-\alpha}
  -\eps X_{-2\alpha-\beta},
&
  [D_4,H_{0,1}]
&=&
  -\eps X_{\alpha+\beta}
  +X_{-\alpha}
\end{array}
\end{equation*}
and 
{\small
\begin{equation*}
  M_\Gamma
=
  \begin{pmatrix}
    -\Gamma &  0      &  0      &  0      &  0      & -\eps   &  0      &  0      &  0      &  0      \\
     0      & -\Gamma &  2      &  0      &  2\eps  &  0      &  0      &  0      &  0      &  0      \\
     0      &  0      & -\Gamma &  2      &  0      &  0      &  0      &  0      & -\eps   & -\eps   \\
    -\eps   &  0      &  0      & -\Gamma &  0      &  0      &  0      &  0      &  0      &  0      \\
     0      &  0      & \eps    &  0      & -\Gamma &  0      &  0      & 2\eps   & -1      &  1      \\
     0      &  0      &  0      &  0      &  0      & -\Gamma &  0      &  0      &  0      &  0      \\
    -\eps   &  0      &  0      &  0      &  0      & -1      & -\Gamma &  0      &  0      &  0      \\
     0      &  0      &  0      &  0      &  0      &  0      & -1      & -\Gamma & -\eps   &  0      \\
    -1      &  0      &  0      & -2\eps  &  0      &  0      & -\eps   &  0      & -\Gamma &  0      \\
     1      &  0      &  0      &  0      &  0      &  0      & -\eps   &  0      &  0      & -\Gamma    
  \end{pmatrix}
\end{equation*}}\noindent
with
$
  \det \, M_\Gamma
=
  \Gamma^{10}.
$
\\[1ex]
\textbf{Case 5.} 
For
$  D_5
=
  \frac{\eps}{2} X_{2\alpha+\beta}
  +
  \frac{b^2\eps}{2} X_{-2\alpha-\beta}
  +
  aH_{0,1}
$
we obtain
\begin{equation*}
\begin{array}{lcllcl}
  [D_5,X_{\alpha}]
&=&
  -aX_{\alpha} + b^2\eps X_{-\alpha-\beta},
	&
[D_5,X_{-\alpha}]
&= &
  \eps X_{\alpha+\beta},	
\\
  {[}D_5,X_{\beta}]
&=&
  2aX_{\beta},
	&
	[D_5,X_{-\beta}]
&=&
  -2aX_{-\beta},
\\
  {[}D_5,X_{\alpha+\beta}]
&=&
  a X_{\alpha+\beta}-b^2\eps X_{-\alpha},
	&
	[D_5,X_{-\alpha-\beta}]
&=&
  -\eps X_{\alpha}-aX_{-\alpha-\beta},
\\
  {[}D_5,X_{2\alpha+\beta}]
&=&
  2b^2\eps H_{1,0},
	&
  [D_5,X_{-2\alpha-\beta}]
&= &
  -2\eps H_{1,0},
\\
  {[}D_5,H_{1,0}]
&=&
  -\eps X_{2\alpha+\beta}  + b^2\eps X_{-2\alpha-\beta},
& 
 [D_5,H_{0,1}],
&=&
  0
\end{array}
\end{equation*}
and
{\small
\begin{equation*}
  M_\Gamma
=
  \begin{pmatrix}
    -a-\Gamma & 0         &  0        &  0        & 0        &  0         & -\eps     &  0      &  0       &  0      \\
     0        & 2a-\Gamma &  0        &  0        & 0        &  0         &  0        &  0      &  0       &  0      \\
     0        & 0         &  a-\Gamma &  0        & \eps     &  0         &  0        &  0      &  0       &  0      \\
     0        & 0         &  0        & -\Gamma   & 0        &  0         &  0        &  0      & -\eps    &  0      \\
     0        & 0         & -b^2\eps  &  0        & a-\Gamma &  0         &  0        &  0      &  0       &  0      \\
     0        & 0         &  0        &  0        & 0        & -2a-\Gamma &  0        &  0      &  0       &  0      \\
     b^2\eps  & 0         &  0        &  0        & 0        &  0         & -a-\Gamma &  0      &  0       &  0      \\
     0        & 0         &  0        &  0        & 0        &  0         &  0        & -\Gamma &  b^2\eps &  0      \\
     0        & 0         &  0        &  2b^2\eps & 0        &  0         &  0        & -2\eps  & -\Gamma  &  0      \\
     0        & 0         &  0        &  0        & 0        &  0         &  0        &  0      &  0       & -\Gamma 
  \end{pmatrix}
\end{equation*}}\noindent
with 
\begin{equation*}
  \det \, M_\Gamma
= 
  \Gamma^2 
  (\Gamma - 2a) 
  (\Gamma + 2a)
  (\Gamma^2 + 4b^2)
  \left((\Gamma -a)^2  + b^2 \right)
  \left((\Gamma^2 + a)^2 + b^2 \right), \quad a,b \ge 0.
\end{equation*}
\textbf{Case 6.}
For
$
  D_6
=
  \eta X_{\beta}
  +
  \frac{\eps}{2}X_{2\alpha+\beta}
  +
  b_2^2\eta X_{-\beta}
  +
  \frac{b_1^2\eps}{2}X_{-2\alpha-\beta}
	$
we obtain
\begin{equation*}
\begin{array}{lcllcl}
  [D_6,X_{\alpha}]
&=&
  -\eta X_{\alpha+\beta} + b_1^2\eps X_{-\alpha-\beta},
	&
	[D_6,X_{-\alpha}]
&=& 
  \eps X_{\alpha+\beta} - b_2^2\eta X_{-\alpha-\beta},
\\
  {[}D_6,X_{\beta}]
&=&
  b_2^2\eta H_{0,1},
	&
	[D_6,X_{-\beta}]
&=&
  -\eta H_{0,1},
\\
  {[}D_6,X_{\alpha+\beta}]
&=&
  b_2^2\eta X_{\alpha} - b_1^2\eps X_{-\alpha},
	&
[D_6,X_{-\alpha-\beta}]
&=&
  -\eps X_{\alpha} + \eta X_{-\alpha},	
\\
  {[}D_6,X_{2\alpha+\beta}]
&=&
  2b_1^2\eps H_{1,0},
&
  [D_6,X_{-2\alpha-\beta}]
&= &
  -2\eps H_{1,0},
\\
  {[}D_6,H_{1,0}]
&=&
  -\eps X_{2\alpha+\beta} + b_1^2\eps X_{-2\alpha-\beta},
&
  [D_6,H_{0,1}]
&= &
  -2\eta X_{\beta} + 2b_2^2\eta X_{-\beta}
\end{array}
\end{equation*}
and
{\small
\begin{equation*}
  M_\Gamma
=
  \begin{pmatrix}
    -\Gamma    &  0         &  b_2^2\eta &  0          &  0         &  0      & -\eps   &  0      &  0         &  0          \\
     0         & -\Gamma    &  0         &  0          &  0         &  0      &  0      &  0      &  0         & -2\eta      \\
    -\eta      &  0         & -\Gamma    &  0          &  \eps      &  0      &  0      &  0      &  0         &  0          \\
     0         &  0         &  0         & -\Gamma     &  0         &  0      &  0      &  0      & -\eps      &  0          \\
     0         &  0         & -b_1^2\eps &  0          & -\Gamma    &  0      &  \eta   &  0      &  0         &  0          \\
     0         &  0         &  0         &  0          &  0         & -\Gamma &  0      &  0      &  0         &  2b_2^2\eta \\
     b_1^2\eps &  0         &  0         &  0          & -b_2^2\eta &  0      & -\Gamma &  0      &  0         &  0          \\
     0         &  0         &  0         &  0          &  0         &  0      &  0      & -\Gamma &  b_1^2\eps &  0          \\
     0         &  0         &  0         &  2b_1^2\eps &  0         &  0      &  0      & -2\eps  & -\Gamma    &  0          \\
     0         &  b_2^2\eta &  0         &  0          &  0         & -\eta   &  0      &  0      &  0         & -\Gamma        
  \end{pmatrix}
\end{equation*}}\noindent
with
\begin{equation*}
  \det \, M_\Gamma
=
  \Gamma^2 
  (\Gamma^2 + 4b_2^2) 
  (\Gamma^2 + 4 b_1^2) 
  \left(\Gamma^2 + (b_1 - b_2)^2 \right) 
  \left(\Gamma^2 + (b_1 + b_2)^2 \right), 
\quad 
  b_1 \ge b_2 \ge 0.
\end{equation*}
\textbf{Case 7.} 
For
$
  D_7
=
  -
  X_{\alpha}
  -
  \eps X_{\beta}
  +
  \frac{\eps}{2b^2}X_{2\alpha+\beta}
  -
  b^2X_{-\alpha}
$
and
\begin{equation*}
\begin{array}{lcllcl}
  [D_7,X_{\alpha}]
&=&
  -\eps X_{\alpha+\beta} - b^2H_{1,0} + b^2H_{0,1}, 
	&
	[D_7,X_{-\alpha}]
&=& 
  \frac{\eps}{b^2} X_{\alpha+\beta}+H_{1,0}-H_{0,1},
\\
  {[}D_7,X_{\beta}]
&=&
  -X_{\alpha+\beta},
	&
	 [D_7,X_{-\beta}]
&=&
  -b^2 X_{-\alpha-\beta} - \eps H_{0,1},
\\
  {[}D_7,X_{\alpha+\beta}]
&=&
  2b^2 X_{\beta}- X_{2\alpha+\beta},
	&
	  [D_7,X_{-\alpha-\beta}]
&=&
  -\frac{\eps}{b^2} X_{\alpha} + \eps X_{-\alpha}
\\
&&&&&
+\, 2X_{-\beta}-b^2 X_{-2\alpha-\beta},
\\
  {[}D_7,X_{2\alpha+\beta}]
&=&
  2b^2 X_{\alpha+\beta},
&
  [D_7,X_{-2\alpha-\beta}]
&=&
  2X_{-\alpha-\beta} - \frac{2\eps}{b^2}H_{1,0},
\\
  {[}D_7,H_{1,0}]
&=&
  X_{\alpha} - \frac{\eps}{b^2} X_{2\alpha+\beta}-b^2X_{-\alpha},
&
  [D_7,H_{0,1}]
&=& 
  -X_{\alpha}  - 2\eps X_{\beta}+b^2X_{-\alpha}
\end{array}
\end{equation*}
and
{\small
\begin{equation*}
  M_\Gamma
=
  \begin{pmatrix}
    -\Gamma    &  0      &  0      &  0      &  0                &  0      & -\frac{\eps}{b^2} &  0                 &  1                & -1      \\
     0         & -\Gamma &  2b^2   &  0      &  0                &  0      &  0                &  0                 &  0                & -2\eps  \\
    -\eps      & -1      & -\Gamma &  2b^2   &  \frac{\eps}{b^2} &  0      &  0                &  0                 &  0                &  0      \\
     0         &  0      & -1      & -\Gamma &  0                &  0      &  0                &  0                 & -\frac{\eps}{b^2} &  0      \\
     0         &  0      &  0      &  0      & -\Gamma           &  0      &  \eps             &  0                 & -b^2              &  b^2    \\
     0         &  0      &  0      &  0      &  0                & -\Gamma &  2                &  0                 &  0                &  0      \\
     0         &  0      &  0      &  0      &  0                & -b^2    & -\Gamma           &  2                 &  0                &  0      \\
     0         &  0      &  0      &  0      &  0                &  0      & -b^2              & -\Gamma            &  0                &  0      \\
    -b^2       &  0      &  0      &  0      &  1                &  0      &  0                & -\frac{2\eps}{b^2} & -\Gamma           &  0      \\
     b^2       &  0      &  0      &  0      & -1                & -\eps   &  0                &  0                 &  0                & -\Gamma     
  \end{pmatrix}
\end{equation*}}
with $\det \, M_\Gamma = \Gamma^4 (\Gamma^2 + 4 b^2)^3$, $b > 0$.


\end{document}